\documentclass[leqno, 12pt]{article}
\usepackage{amsmath,amsfonts,amsthm,amssymb,indentfirst,mathrsfs}
\usepackage{xy} \xyoption{all}

\newtheorem{theorem}{Theorem}
\newtheorem{lemma}[theorem]{Lemma}

\newtheorem{corollary}[theorem]{Corollary}
\newtheorem{proposition}[theorem]{Proposition}

\theoremstyle{definition}
\newtheorem{example}[theorem]{Example}

\newcommand{\path}{\mathrm{Path}}
\newcommand{\R}{\mathbb{R}}

\newcommand{\N}{\mathbb{N}}
\newcommand{\M}{\mathbb{M}}

\begin{document}

\title{Topological Graph Inverse Semigroups}
\author{Z. Mesyan, J. D. Mitchell, M. Morayne, Y. H. P\'eresse}

\maketitle

\begin{abstract}

To every directed graph $E$ one can associate a \emph{graph inverse semigroup} $G(E)$, where elements roughly correspond to possible paths in $E$. These semigroups generalize polycyclic monoids, and they arise in the study of Leavitt path algebras, Cohn path algebras, graph $C^*$-algebras, and Toeplitz $C^*$-algebras. We investigate topologies that turn $G(E)$ into a topological semigroup. For instance, we show that in any such topology that is Hausdorff, $G(E)\setminus \{0\}$ must be discrete for any directed graph $E$. On the other hand, $G(E)$ need not be discrete in a Hausdorff semigroup topology, and for certain graphs $E$, $G(E)$ admits a $T_1$ semigroup topology in which $G(E)\setminus \{0\}$ is not discrete. We also describe, in various situations, the algebraic structure and possible cardinality of the closure of $G(E)$ in larger topological semigroups.
\medskip



\end{abstract}

\section{Introduction}

Given any directed graph $E$, one can construct a \emph{graph inverse semigroup} $G(E)$ (to be defined precisely below), where vaguely speaking, the elements correspond to possible paths in $E$. This class of semigroups was first introduced by Ash/Hall~\cite{AH}, in order to show that every partially ordered set can be realized as the partially ordered set of nonzero $\mathscr{J}$-classes of a semigroup. (Two elements in a semigroup are $\mathscr{J}$-\emph{equivalent} if they generate the same ideal.) \emph{Polycyclic monoids}, which were first defined by Nivat/Perrot~\cite{NP}, are a particularly well-studied class of graph inverse semigroups (e.g., \cite{Lawson,MS}). These monoids (with zero) have presentations by generators and relations of the following form: $$P_n = \langle e_1, \dots, e_n, e_1^{-1}, \dots, e_n^{-1} : e_i^{-1}e_j = \delta_{ij}\rangle,$$ where $\delta_{ij}$ is the Kronecker delta, and $P_1$ is known as the \emph{bicyclic monoid}. Graph inverse semigroups also arise in the study of rings and $C^*$-algebras. More specifically, for any field $K$ and any directed graph $E$, the (contracted) semigroup ring $KG(E)$ is called the \emph{Cohn path $K$-algebra} of $E$. The quotient of a Cohn path algebra by a certain ideal is known as the \emph{Leavitt path $K$-algebra} of $E$. These rings were introduced independently by Abrams/Aranda Pino~\cite{AP} and Ara/Moreno/Pardo~\cite{AMP}, and they have attracted much attention in recent years. Cohn path algebras and Leavitt path algebras are algebraic analogues of Toeplitz $C^*$-algebras and graph $C^*$-algebras (see~\cite{KPRR, KPR}), respectively. For more on the connection between graph inverse semigroups and $C^*$-algebras see~\cite{Paterson}. These semigroups have also received much attention in their own right recently~\cite{CS, JL, Krieger, Lawson2, Lenz, MM}.

In this article we study topologies that turn the graph inverse semigroups $G(E)$ into topological semigroups, i.e., topologies on $G(E)$ in which the multiplication operation of $G(E)$ is continuous. We show, among other things, that with respect to any such Hausdorff topology, $G(E)\setminus \{0\}$ must be discrete for all directed graphs $E$ (Theorem~\ref{discrete}), but that $G(E)$ admits a non-discrete metrizable semigroup topology for ``most" choices of $E$ (Proposition~\ref{non-disc-top}). Moreover, for certain directed graphs $E$, $G(E)$ admits a $T_1$ semigroup topology in which $G(E)\setminus \{0\}$ is not discrete (Example~\ref{T1-eg}). However, if $E$ is a finite graph, then the only locally compact Hausdorff semigroup topology on $G(E)$ is the discrete topology (Theorem~\ref{loc-comp}). We also show, for any $E$, that if $\overline{G(E)}$ is the closure of $G(E)$ in a Hausdorff topological inverse semigroup (i.e., one where in addition to the multiplication, the inversion operation is continuous), and $\mu \in \overline{G(E)}\setminus G(E)$ is any idempotent, then $\mu (\overline{G(E)}\setminus G(E))\mu \setminus \{0\}$ is either the trivial group or a group that contains a dense cyclic subgroup (Theorem~\ref{comp-gp-thrm}). Along the way to proving this result, we characterize all inverse subsemigroups $S$ of $G(E)$ such that $\mu\nu \neq 0$ for all $\mu, \nu \in S \setminus \{0\}$ (Theorem~\ref{no_zero_div_description}). In the final section, we give several results about the possible cardinalities of the complements of polycyclic monoids in their closures within larger topological semigroups.

Some of the aforementioned results generalize similar facts about the bicyclic monoid proved by Eberhart/Selden~\cite{ES}, though our proofs are generally quite different.

\subsection*{Acknowledgements}
We would like to thank the referee for pointing us to related literature and for suggesting very interesting directions for future research. We are also grateful to the University of St Andrews and the Wroc{\l}aw University of Technology for their hospitality during the writing of this article. Micha{\l} Morayne was partially supported by NCN grant DEC-2011/01/B/ST1/01439 while this work was performed. 

\section{Preliminaries}\label{definitions}

\subsection{Semigroups and Topology}\label{top-sect}

A semigroup $S$ is an \emph{inverse semigroup} if for each $x \in S$ there is a unique element $x^{-1} \in S$ satisfying $x = xx^{-1}x$ and $x^{-1} = x^{-1}xx^{-1}$. If $S$ is a semigroup and $\mathcal{O}$ is a topology on $S$, then we say that $S$ is a $\emph{topological semigroup}$ with respect to $\mathcal{O}$, or that $\mathcal{O}$ is a \emph{semigroup topology} on $S$, if the multiplication operation $* : S \times S \to S$ on $S$ is continuous with respect to $\mathcal{O}$, where $S\times S$ is endowed with the product topology.  If $S$ is an inverse semigroup that is a topological semigroup, then  $S$ is a $\emph{topological inverse semigroup}$ if the inverse operation $\cdot^{-1} : S \to S$ on $S$ is continuous.

Next we recall some standard topological concepts and notation. Let $X$ be a topological space. Then we say that $X$ is $T_1$ if for any two points $x, y \in X$ there is an open neighborhood of $x$ that does not contain $y$. Also, $X$ is said to be $T_2$, or \emph{Hausdorff}, if for any two points $x,y \in X$ there are open neighborhoods $U$ and $V$ of $x$ and $y$, respectively, such that $U \cap V = \emptyset$. If $Y \subseteq X$, then we denote the closure of $Y$ in $X$ by $\overline{Y}$. Also, if $d : X \times X \to \R$ is a metric (where $\R$ is the set of the real numbers), $x \in X$, and $m > 0$, then we let $B(x,m) = \{y \in X : d(x,y) < m\}$.

A basic fact about topological semigroups that will be useful is that if $G$ is a semigroup with zero element $0$, and $G$ is dense in a larger $T_1$ topological semigroup $S$, then $0 \cdot \mu=0$ and $\mu \cdot 0=0$ for all $\mu \in S$. To show the first equality (the second follows similarly), suppose that $0 \cdot \mu \neq 0$. Since the topology is $T_1$, there must be an open neighborhood $U$ of $0 \cdot \mu$ such that $0 \notin U$. By the continuity of multiplication, we can then find an open neighborhood $W$ of $\mu$ such that $0 \cdot W \subseteq U$. But, since $G$ is dense in $S$, $W$ must contain some $\nu \in G$, implying that $0\cdot \nu = 0 \in U$, contrary to assumption. Hence $0 \cdot \mu=0$. 

We shall denote the cardinality of a set $X$ by $|X|$. The set of all natural numbers (including $0$) will be denoted by $\N$.

\subsection{Graphs}

A \emph{directed graph} $E=(E^0,E^1,r,s)$ consists of two sets $E^0,E^1$ (containing \emph{vertices} and \emph{edges}, respectively), together with functions $s,r:E^1 \to E^0$, called \emph{source} and \emph{range}, respectively. A \emph{path} $x$ in $E$ is a finite sequence of (not necessarily distinct) edges $x=e_1\dots e_n$ such that $r(e_i)=s(e_{i+1})$ for $i=1,\dots,n-1$; in this case, $s(x):=s(e_1)$ is the \emph{source} of $x$, $r(x):=r(e_n)$ is the \emph{range} of $x$, and $|x|:=n$ is the \emph{length} of $x$. If $x = e_1\dots e_n$ is a path in $E$ such that $s(x)=r(x)$ and $s(e_i)\neq s(e_j)$ for every $i\neq j$, then $x$ is called a \emph{cycle}.  A cycle consisting of one edge is called a \emph{loop}. We view the elements of $E^0$ as paths of length $0$ (extending $r$ and $s$ to $E^0$ via $r(v)=v = s(v)$ for all $v\in E^0$), and denote by $\path(E)$ the set of all paths in $E$. A directed graph for which both $E^0$ and $E^1$ are finite sets is called a \emph{finite directed graph}.  From now on we shall refer to directed graphs as simply ``graphs".

Given a graph $E=(E^0,E^1,r,s)$, the \emph{graph inverse semigroup $G(E)$ of $E$} is the semigroup with zero generated by the sets $E^0$ and $E^1$, together with a set of variables $\{e^{-1} : e\in E^1\}$, satisfying the following relations for all $v,w\in E^0$ and $e,f\in E^1$:\\
(V)  $vw = \delta_{v,w}v$,\\ 
(E1) $s(e)e=er(e)=e$,\\
(E2) $r(e)e^{-1}=e^{-1}s(e)=e^{-1}$,\\
(CK1) $e^{-1}f=\delta _{e,f}r(e)$.\\
We define $v^{-1}=v$ for each $v \in E^0$, and for any path $y=e_1\dots e_n$ ($e_1\dots e_n \in E^1$) we let $y^{-1} = e_n^{-1} \dots e_1^{-1}$. With this notation, every nonzero element of $G(E)$ can be written uniquely as $xy^{-1}$ for some $x, y \in \path(E)$, by the CK1 relation. It is also easy to verify that $G(E)$ is indeed an inverse semigroup, with $(xy^{-1})^{-1} = yx^{-1}$ for all $x, y \in \path (E)$.

Informally speaking, we start with a graph $E$, add for each edge $e\in E^1$ a ``ghost" edge $e^{-1}$ going in the opposite direction of $e$ (between the same two vertices), and then turn $E$ into a semigroup where the elements correspond to possible paths in our extended graph. (Products of edges that do not occur consecutively along a possible path are $0$.)

If $E$ is a graph having only one vertex $v$ and $n$ edges (necessarily loops), for some integer $n \geq 1$, then $G(E)$ is known as a \emph{polycyclic monoid}, and we shall denote it by $P_n$. In particular, $P_n$ can be viewed as the monoid with zero presented by $$\langle e_1, \dots, e_n, e_1^{-1}, \dots, e_n^{-1} : e_i^{-1}e_j = \delta_{ij}\rangle,$$ if we identify the one vertex in $E$ with the identity element $1$ of this monoid. We note that the \emph{bicyclic monoid} $P_1$, as traditionally discussed in the literature, does not have a zero element. To allow for uniformity of treatment, however, we shall assume that a zero element has been adjoined, whenever referring to $P_1$.

\subsection{Connections With Rings}

In the Introduction we mentioned that graph inverse semigroups arise in the study of certain rings and $C^*$-algebras. Having defined these semigroups, we can state their connection with the rings in question explicitly. The reader unfamiliar with rings may safely skip this subsection.
 
Let $K$ be a field and $E$ a graph. Then the contracted semigroup ring $\overline{KG(E)}$ (i.e., the semigroup ring resulting from identifying the zero element of $G(E)$ with the zero in the semigroup ring $KG(E)$) is known as the \emph{Cohn path $K$-algebra} $C_K(E)$ \emph{of $E$}. Letting $N$ denote the ideal of $C_K(E)$ generated by elements of the form $v~-~\sum_{e\in s^{-1}(v)} ee^*,$ where $v \in E^0$ is a regular vertex (i.e., one that emits a nonzero finite number of edges), the ring $C_K(E)/N$ is called the \emph{Leavitt path $K$-algebra} $L_K(E)$ \emph{of $E$}. 

Many well-known rings arise as Leavitt path algebras. For example, the classical Leavitt $K$-algebra $L_K(n)$ for $n\geq 2$, introduced in~\cite{Leavitt} (which is universal with respect to an isomorphism property between finite-rank free modules), can be expressed as the Leavitt path algebra of the ``rose with $n$ petals" graph pictured below.
$$
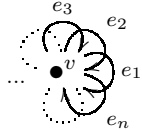
\xymatrix{ & {\bullet^v} \ar@(ur,dr) ^{e_1} \ar@(u,r) ^{e_2}
\ar@(ul,ur) ^{e_3} \ar@{.} @(l,u) \ar@{.} @(dr,dl) \ar@(r,d) ^{e_n}
\ar@{}[l] ^{\ldots} }$$
The full $d\times d$ matrix algebra $\M_d(K)$ is isomorphic to the Leavitt path algebra of the oriented line graph with $d$ vertices, shown below.
$$\xymatrix{{\bullet}^{v_1} \ar [r] ^{e_1} & {\bullet}^{v_2}  \ar@{.}[r] & {\bullet}^{v_{d-1}} \ar [r]^{e_{d-1}} & {\bullet}^{v_d}}$$
Also, the Laurent polynomial algebra $K[x,x^{-1}]$ can be identified with the Leavitt path algebra of the following ``one vertex, one loop" graph.
$$
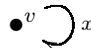
\xymatrix{{\bullet}^{v} \ar@(ur,dr) ^x}$$

\section{Hausdorff Topologies}

Our first goal is to show that, with the possible exception of $0$, every element of $G(E)$ must be isolated in any Hausdorff semigroup topology on $G(E)$. We begin with two lemmas.

\begin{lemma} \label{1-1_lemma}
The following hold for any graph $E$.
\begin{enumerate}
\item[$(1)$] If $x,y \in \path(E)$ are such that $r(x)=r(y)=v$, then $G(E)xy^{-1}G(E) = G(E)vG(E).$
\item[$(2)$] For any $\mu, \nu \in G(E)\setminus \{0\}$, the sets $\, \{\rho \in G(E) : \mu\rho=\nu\}$ and $\, \{\rho \in G(E) : \rho\mu=\nu\}$ are finite.
\end{enumerate}
\end{lemma}

\begin{proof}
(1) This follows from the equations $x^{-1}(xy^{-1})y = v^2 = v$ and $xy^{-1}=xvy^{-1}$.

\vspace{\baselineskip}

(2) Write $\mu=pq^{-1}$, $\nu=tu^{-1}$, and $\rho=xy^{-1}$, where $p,q,t,u,x,y \in \path (E)$. If $\mu\rho=\nu$, then there must be a path $z \in \path (E)$ such that either $x=qz$ or $q=xz$ (for, otherwise $q^{-1}x=0$). In the first case, $$\mu\rho=pq^{-1}xy^{-1}=pq^{-1}qzy^{-1}=pzy^{-1},$$ implying that $t=pz$ and $u=y$, which determines $\rho$ uniquely as $\rho = qzu^{-1} = qp^{-1}tu^{-1}$. In the second case, $$\mu\rho=pq^{-1}xy^{-1}=pz^{-1}x^{-1}xy^{-1}=pz^{-1}y^{-1},$$ implying that $t=p$ and $u=yz$. That $\{\rho \in G(E) : \mu\rho=\nu\}$ is finite now follows from the fact that only finitely many choices of $x, y, z$ can satisfy $q=xz$ and $u=yz$.

The finiteness of $\{\rho \in G(E) : \rho\mu=\nu\}$ can be shown by a similar argument.
\end{proof}

\begin{lemma} \label{1-2_lemma}
Let $E$ be a graph, and suppose that $G(E)$ is a topological semigroup with respect to a $T_1$ topology $\mathcal{O}$.
\begin{enumerate}
\item[$(1)$] Suppose that $\mu \in G(E)\setminus \{0\}$ is a limit point, and let $\nu \in G(E)$. If $\mu\nu \neq 0$, then $\mu\nu$ is a limit point, and if $\nu\mu \neq 0$, then $\nu\mu$ is a limit point.
\item[$(2)$] If $v \in E^0$ is a limit point, then $\, |\{e \in E^1 : s(e) = v\}|=1$.
\item[$(3)$] Suppose that $v \in E^0$ is a limit point and $z=e_1\dots e_n \in \path (E)$ is a cycle $(e_1, \dots, e_n \in E^1)$ such that $s(z)=v=r(z)$. Then there is an open neighborhood of $v$ consisting entirely of elements of the form $z^le_1\dots e_ke_k^{-1}\dots e_1^{-1}z^{-m}$ $($$k,l,m \in \N$, $k < n$$)$, where we understand $z^0$ to be $v$ and $z^le_1\dots e_ke_k^{-1}\dots e_1^{-1}z^{-m} = z^lz^{-m}$ when $k=0$.
\item[$(4)$] Suppose that $v \in E^0$ is a limit point, and for all $p \in \path (E)$ with source $v$, $r(p)$ is not the source of a cycle. Then there exists a sequence $e_1, e_2, \dots $ of edges and an open neighborhood of $v$ consisting entirely of nonzero elements of the form $e_1 \dots e_ne_n^{-1} \dots e_1^{-1}$ $($$n \in \N$$)$, where we understand $e_1 \dots e_ne_n^{-1} \dots e_1^{-1}$ to be $v$ when $n=0$.
\end{enumerate}
\end{lemma}

\begin{proof}
(1) Suppose that $\mu\nu \neq 0$, and let $U$ be an open neighborhood of $\mu\nu$. By the continuity of multiplication, there is an open neighborhood $V$ of $\mu$ such that $V\nu \subseteq U$. Since $\mu$ is a limit point, $V$ must be infinite, and hence $U$ must be infinite as well, by Lemma~\ref{1-1_lemma}(2). Thus, $\mu\nu$ is also a limit point, and, by a similar argument, so is $\nu\mu$ (in case it is nonzero).

\vspace{\baselineskip}

(2) First suppose that $v$ is a sink, i.e.\ that $|\{e \in E^1 : s(e) = v\}|=0$. Since $\mathcal{O}$ is $T_1$, we can find an open neighborhood $U$ of $v$ such that $0 \notin U$. Since $vv=v$, by the continuity of multiplication, there must be an (infinite) open neighborhood $V$ of $v$ such that $VV \subseteq U$. But, since $v$ is a sink, for any $\mu \in G(E) \setminus \{v\}$ either $v\mu = 0$ or $\mu v = 0$. Hence either $0 \in \mu V$ or $0 \in V \mu$, which implies that $0 \in U$, contrary to our assumption. Thus $v$ cannot be a sink.

Now suppose that  $|\{e \in E^1 : s(e) = v\}|\geq 2$. Suppose also that there is an open neighborhood $U$ of $v$ and an edge $e \in E^1$ such that for all $\mu \in U \setminus \{v\}$, $\mu = e\nu$ for some $\nu \in G(E)$. By assumption, we can find some $f \in E^1 \setminus \{e\}$ such that $s(f) = v$, and since the topology is $T_1$, there is an open neighborhood $W$ of $f^{-1}$ such that $0 \notin W$. Since $f^{-1}v=f^{-1}$ and multiplication is continuous, there must be an open neighborhood $U'$ of $v$ such that $f^{-1}U' \subseteq W$, and hence $f^{-1}(U \cap U') \subseteq W$. But, for all  $\mu \in U \setminus \{v\}$ we have $0 = f^{-1}\mu$, by our choice of $U$, and hence $0 \in W$ (since $U \cap U'$ is infinite, by the fact that $v$ is a limit point), contrary to assumption. Thus for any open neighborhood $U$ of $v$ and $e \in E^1$ there is some $\mu \in U \setminus \{v\}$ such that $\mu \neq e\nu$ for all $\nu \in G(E)$.

Next, let $e, f \in E^1$ be distinct edges having source $v$. Since $\mathcal{O}$ is $T_1$, we can find open neighborhoods $U_e$, $U_f$, $U_{e^{-1}}$ of $e$, $f$, and $e^{-1}$, respectively, such that $0 \notin U_e \cup U_f \cup U_{e^{-1}}$. Since $ve=e$, $vf=f$, and $e^{-1}v = e^{-1}$, by the continuity of multiplication we can find open neighborhoods $V_e$, $V_f$, and $V_{e^{-1}}$ of $v$ such that $V_ee \subseteq U_e$, $V_ff \subseteq U_f$, and $e^{-1}V_{e^{-1}} \subseteq U_{e^{-1}}$. Let $V = V_e \cap V_f \cap V_{e^{-1}}$. Then $Ve \subseteq U_e$, $Vf \subseteq U_f$, and $e^{-1}V \subseteq U_{e^{-1}}$, which implies that $0 \notin Ve \cup Vf \cup e^{-1}V$. Since $e \neq f$, $0 \notin Ve \cup Vf$ implies that $V \subseteq \path (E)$. Since $V$ is infinite, by the previous paragraph, this in turn implies that there is some $\mu \in V \setminus \{v\}$ such that $\mu \in \path (E)$ and $\mu \neq e\nu$ for all $\nu \in G(E)$. But then $e^{-1}\mu = 0$, and hence $0 \in e^{-1}V \subseteq U_{e^{-1}}$, contradicting our choice of $U_{e^{-1}}$. Therefore we cannot have $|\{e \in E^1 : s(e) = v\}|\geq 2$, and hence $|\{e \in E^1 : s(e) = v\}|=1$.

\vspace{\baselineskip}

(3) Let $v \in E^0$ and $z = e_1\dots e_n$ be a cycle as in the statement. By (2), we have $|\{e \in E^1 : s(e) = v\}|=1$. We note also that if $p \in \path (E)$ is such that $s(p) = v$, then $r(p)$ must be a limit point. For, since $p=vp$, by (1), $p$ must be a limit point, and since $r(p) = p^{-1}p$, by the same statement, $r(p)$ must be a limit point. It follows, by (2), that the only paths having source $v$ are of the form $z^le_1\dots e_k$ ($k,l \in \N$, $k\leq n$). Since $\mathcal{O}$ is $T_1$, we can find an open neighborhood $V$ of $v$ such that $0 \notin V$. Since $vv=v$, by the continuity of multiplication, we can find an open neighborhood $W$ of $v$ such that $WW \subseteq V$. Thus, letting $\mu \in W$ be any element, we see that $v \mu \neq 0 \neq \mu v$. Writing $\mu = pq^{-1}$ for some $p,q \in \path (E)$, it follows that $s(p)=v=s(q)$. But, by the above description of such paths, this means that $p=z^le_1\dots e_k$ and $q=z^me_1\dots e_j$ for some $j,k,l,m \in \N$, with $j,k < n$. Since $\mu = pq^{-1} \neq 0$, it must be the case that $j=k$, and hence $p=z^le_1\dots e_k$ and $q=z^me_1\dots e_k$. Therefore $W$ consists entirely of elements of the desired form.

\vspace{\baselineskip}

(4) Suppose that $v=v_1 \in E^0$ is a limit point. By (2), there is a unique edge $e_1 \in E^1$ having source $v_1$. By hypothesis, $e_1$ is not a loop, and hence $v_2 = r(e_1) \neq v_1$.  Since $e_1=v_1e_1$, by (1), $e_1$ must be a limit point, and since $v_2 = e_1^{-1}e_1$, $v_2$ must be a limit point. Letting $e_2 \in E^1$ be the unique edge having source $v_2$, by hypothesis, $r(e_1e_2) \notin \{v_1, v_2\}$. Repeating this argument, we conclude that $E$ must have the following subgraph, where for each $i\geq 1$, $e_i \in E^1$ is the only edge with source $v_i$.
$$\xymatrix{{\bullet}^{v_1} \ar [r] ^{e_1} & {\bullet}^{v_2} \ar [r] ^{e_2} & {\bullet}^{v_3} \ar [r] ^{e_3}  & ...  }$$
It follows that the only paths having source $v_1$ are of the form $e_1 \dots e_n$ ($n \in \N$).

By the same argument as in the proof of (3), we can find an open neighborhood $W$ of $v_1$ such that $v\mu \neq 0 \neq \mu v$ for every $\mu \in W$. Thus, for every $\mu \in W$, writing $\mu = pq^{-1}$ ($p,q \in \path (E)$), we must have $s(p)=v=s(q)$. Therefore, $p=q=e_1 \dots e_n$ for some $n \in \N$ (since $r(p)=r(q)$), and hence $W$ consists entirely of elements of the form $e_1 \dots e_ne_n^{-1} \dots e_1^{-1}$ ($n \in \N$), as desired.
\end{proof}

We are now ready to prove the main result of this section. It generalizes \cite[Corollary I.2]{ES}, which says that the bicyclic monoid (without zero) is discrete in any semigroup topology, though our proof uses a very different approach.

\begin{theorem} \label{discrete}
Suppose that $E$ is a graph, and that $G(E)$ is a topological semigroup with respect to a Hausdorff topology $\mathcal{O}$. Then $G(E)\setminus \{0\}$ must be discrete.
\end{theorem}

\begin{proof}
Suppose that $\mu \in G(E)\setminus \{0\}$ is a limit point. By Lemma~\ref{1-1_lemma}(1), we can find some $\nu, \tau \in G(E)$ such that $\nu\mu\tau \in E^0$. Hence, by Lemma~\ref{1-2_lemma}(1), there must be some vertex $v \in E^0$ that is a limit point. We shall show that this leads to a contradiction.

If there is some $p \in \path (E)$ such that $s(p) =v$ and $r(p)$ is the source of a cycle, then by Lemma~\ref{1-2_lemma}(1), $r(p)$ must be a limit point (since $r(p)=p^{-1}vp$). Thus, upon replacing $v$ with $r(p)$ if necessary, we may assume that either $v$ is the source of a cycle, or that for all $p \in \path (E)$ with $s(p) =v$, $r(p)$ is not the source of a cycle. Then, in either case, by Lemma~\ref{1-2_lemma}(3,4), there is an edge $e_1 \in E^1$ and an open neighborhood $W$ of $v$ such that every element of $W \setminus \{v\}$ either begins with $e_1$ or ends with $e_1^{-1}$. Since $\mathcal{O}$ is Hausdorff, we can find open neighborhoods $V$ and $U$, of $v$ and $e_1e_1^{-1}$, respectively, such that $U \cap V = \emptyset$. Since $e_1e_1^{-1}v=e_1e_1^{-1}$ and $ve_1e_1^{-1}=e_1e_1^{-1}$, upon passing to a subneighborhood of $V$, we may assume that $e_1e_1^{-1}V, Ve_1e_1^{-1} \subseteq U$, using the continuity of multiplication. Upon intersecting with $W$, we may assume that every element of $V \setminus \{v\}$ either begins with $e_1$ or ends with $e_1^{-1}$ (and since $v$ is a limit point, $V$ must be infinite). But then, taking any $\mu \in V\setminus \{v\}$, either $e_1e_1^{-1}\mu = \mu$ or $\mu e_1e_1^{-1}=\mu$, which violates $U \cap V = \emptyset$, giving the desired contradiction.
\end{proof}

Let us next give an example showing that the conclusion of Theorem~\ref{discrete} no longer holds if the Hausdorff assumption is dropped. More specifically, we shall construct a graph $E$ and a $T_1$ (but not Hausdorff) topology, with respect to which $G(E)$ is a topological semigroup and $G(E)\setminus \{0\}$ is not discrete. 

\begin{example} \label{T1-eg}
Let $E$ be the following graph.
$$\xymatrix{{\bullet}^{v_1} \ar [r] ^{e_1} & {\bullet}^{v_2} \ar [r] ^{e_2} & {\bullet}^{v_3} \ar [r] ^{e_3}  & ...  }$$ 
For all $n \in \N$, and for all $p, q \in \path (E)$ such that $r(p)=r(q)$, let $$U_{pq^{-1},n} = \{pxx^{-1}q^{-1} : x \in \path (E), s(x)=r(p)=r(q), |x| > n\} \cup \{pq^{-1}\}.$$ Let $\mathcal{O}$ be the topology on $G(E)$ generated by the subbase consisting of $\{0\}$ and the sets $U_{pq^{-1},n}$. We claim that with this topology $G(E)$ is a $T_1$ topological semigroup, and that $G(E) \setminus \{0\}$ is not discrete.

To show that $\mathcal{O}$ is $T_1$, let $p,q,t,z \in \path (E)$ be such that $pq^{-1} \neq tz^{-1}$ (and $r(p)=r(q)$, $r(t)=r(z)$). Also, let $n \in \N$ be such that $n \geq \mathrm{max}\{|t|,|z|\}$. Then $tz^{-1} \notin U_{pq^{-1},n}$, giving the desired conclusion.

To prove that $G(E) \setminus \{0\}$ is not discrete we shall show that any nonempty finite intersection of sets (other than $\{0\}$) in our subbase contains infinitely many elements. Thus let $n_1, \dots, n_m \in \N$ and $p_1, \dots, p_m, q_1, \dots, q_m \in \path (E)$ be such that $r(p_i)=r(q_i)$ for each $i$, and suppose that $U_{p_1q_1^{-1},n_1} \cap \dots \cap U_{p_mq_m^{-1},n_m} \neq \emptyset$. Then there are $x_1, \dots, x_m \in \path (E)$ such that $p_1x_1x_1^{-1}q_1^{-1} = \dots = p_mx_mx_m^{-1}q_m^{-1},$ where $p_ix_ix_i^{-1}q_i^{-1} \in  U_{p_iq_i^{-1},n_i}$ for each $i$. It follows that $r(p_1x_1) = \dots = r(p_mx_m)$, and therefore if we take any $y \in \path (E)$ such that $s(y) = r(p_1x_1)$ and $|y| > \mathrm{max}\{n_1, \dots, n_m\}$, then $$p_1x_1yy^{-1}x_1^{-1}p_1^{-1} \in \bigcap_{i=1}^m U_{p_iq_i^{-1},n_i}.$$ But, by our choice of $E$, there are infinitely many possibilities for $y$, giving the required conclusion.

It remains to show that $G(E)$ is a topological semigroup, i.e.\ that multiplication is continuous in $\mathcal{O}$. Thus let $\mu, \nu \in G(E)$, and let $U$ be an open neighborhood of $\mu\nu$. We wish to find open neighborhoods $V$ and $W$ of $\mu$ and $\nu$, respectively, such that $VW \subseteq U$. If either of $\mu$ or $\nu$ is $0$, then taking the corresponding neighborhood to be $\{0\}$, the desired result is clear. Thus let us assume that $0\neq \mu = pq^{-1}$ and $0\neq \nu = tz^{-1}$ for some $p,q,t,z \in \path (E)$. If $\mu\nu = 0$, then $U_{pq^{-1},0}U_{tz^{-1},0} = \{0\}$, again gives the desired result. Let us therefore suppose that $\mu\nu \neq 0$, in which case we may also assume that $U = U_{ab^{-1},n}$ for some $n \in \N$ and $a,b \in \path (E)$. Then there must be some $x \in \path (E)$ such that either $q=tx$ or $t=qx$. Let us assume that the latter holds, as the former case can be handled similarly. Thus $$\mu\nu = pq^{-1}tz^{-1} = pq^{-1}qxz^{-1} = pxz^{-1}.$$ To conclude our construction, it is enough to show that $\tau\theta \in U_{pxz^{-1},n} \, (=U_{ab^{-1},n})$ for all $\tau \in U_{pq^{-1},n}$ and $\theta \in U_{tz^{-1},n}$. Write $\tau = pyy^{-1}q^{-1}$ and $\theta = tww^{-1}z^{-1}$ for some $y,w \in \path (E)$, where $|y|$ and $|w|$ are each either $0$ or greater than $n$. Then $$\tau\theta = pyy^{-1}q^{-1}tww^{-1}z^{-1} = pyy^{-1}xww^{-1}z^{-1}.$$ Since $\mu\nu \neq 0$, we also have $\tau\theta \neq 0$, and thus, either $y=xwv$ or $xw=yv$ for some $v \in \path (E)$. Again, let us assume that $xw=yv$, since the other case can be dispatched similarly. Thus, $$\tau\theta = pyy^{-1}yvw^{-1}z^{-1} = pyvw^{-1}z^{-1} = pxww^{-1}z^{-1}\in U_{pxz^{-1},n },$$ as desired, since $|w|$ is either $0$ or greater than $n$.
\end{example}

The next result is a generalization of \cite[Theorem I.3]{ES} to arbitrary graph inverse semigroups.

\begin{theorem} \label{complement_ideal}
Let $E$ be a graph, and suppose that $G(E)$ is a dense subsemigroup of a Hausdorff topological semigroup $S$. Then the following hold.
\begin{enumerate}
\item[$(1)$] $G(E) \setminus \{0\}$ is open in $S$.
\item[$(2)$] $(S\setminus G(E)) \cup \{0\}$ is an ideal of $S$.
\end{enumerate}
\end{theorem}

\begin{proof}
(1) By Theorem~\ref{discrete}, the topology on $G(E) \setminus \{0\}$ inherited from $S$ must be discrete. Thus for any $\mu \in G(E)\setminus \{0\}$, there must be an open subset $U \subseteq S$ such that $U \cap (G(E)\setminus \{0\}) = \{\mu\}$. Since $G(E)$ is dense in $S$, $U \cap G(E)$ is dense in the closure $\overline{U}$ of $U$ in $S$. Hence $\overline{U} = \overline{U \cap G(E)}\subseteq \{\mu, 0\}$, and so either $U = \{\mu\}$ or $U = \{\mu, 0\}$. But, since the topology on $S$ is Hausdorff, in either case we conclude that $\{\mu\}$ is open in $S$, from which the statement follows.

\vspace{\baselineskip}

(2) Let $\mu \in (S\setminus G(E)) \cup \{0\}$ and $\nu \in S$ be any elements. We wish to show that $\mu\nu \in (S\setminus G(E)) \cup \{0\}$ (that $\nu\mu \in (S\setminus G(E)) \cup \{0\}$ can be shown similarly). We may assume that $\mu \neq 0 \neq \nu$, since otherwise the claim follows from the fact that $0\cdot S = \{0\} = S \cdot 0$, mentioned in Subsection~\ref{top-sect}. Seeking a contradiction, suppose that $\mu\nu \in G(E)\setminus \{0\}$. Since $G(E)\setminus \{0\}$ is open (by (1)) and hence discrete in $S$ (by Theorem~\ref{discrete}), we can find open neighborhoods $U$ of $\mu$ and $V$ of $\nu$ such that $UV =\{\mu\nu\}$. Also, $U\cap G(E)$ must be infinite, since $\mu$ is a limit point of $G(E)\setminus \{0\}$, and there must be a point $\tau \in V \cap (G(E)\setminus \{0\})$. Hence $$\{\mu\nu\} = UV \supseteq (U \cap (G(E)\setminus \{0\}))\cdot (V \cap (G(E)\setminus \{0\})) \supseteq (U \cap (G(E)\setminus \{0\})) \tau \neq \emptyset,$$ which implies that $(U \cap (G(E)\setminus \{0\})) \tau = \{\mu\nu\}$, in contradiction to Lemma~\ref{1-1_lemma}(2). Therefore, $\mu\nu \in (S\setminus G(E)) \cup \{0\}$, as desired.
\end{proof}

We next show that while $G(E)\setminus \{0\}$ must be discrete in any Hausdorff semigroup topology, $G(E)$ admits a non-discrete metrizable topology, as long as $E$ has paths of arbitrary length. (This is the case, for instance, for any graph having cycles or an infinite path.)

\begin{proposition} \label{non-disc-top}
Let $E$ be a graph having paths of arbitrary $($finite$)$ length, define $d'(0,0)=0$ and $$d'(pq^{-1},0) = d'(0, pq^{-1}) = \frac{1}{\mathrm{min}\{|p|,|q|\} + 1}$$ for all $p,q \in \path (E)$, and extend $d'$ to a map $d: G(E) \times G(E) \to \R$ via $$d(\mu, \nu) = \left\{ \begin{array}{ll}
d'(\mu, 0) + d'(\nu, 0) & \text{if } \, \mu \neq \nu \\
0 & \text{if } \, \mu = \nu
\end{array}\right..$$ Then $d$ is a metric that induces a non-discrete semigroup topology on $G(E)$.
\end{proposition}

\begin{proof}
It is clear that for all $\mu, \nu \in G(E)$ we have $d(\mu, \nu) \geq 0$,  $d(\mu, \nu) = d(\nu, \mu)$, and  $d(\mu, \nu) = 0$ if only if $\mu =\nu$. It is also easy to see that $d$ satisfies the triangle inequality, and hence $d$ is a metric on $G(E)$. As $E$ has paths of arbitrary length, the topology on $G(E)$ induced by $d$ is not discrete, since $B(0, 1/n) \setminus \{0\} \neq \emptyset$ for every positive $n \in \N$. To verify that this topology is a semigroup topology, let $\mu, \nu \in G(E)$ be any elements, and let $U$ be an open neighborhood of $\mu\nu$. We may assume that $U = B(\mu\nu, 1/n)$ for some positive $n \in \N$, and we wish to find open neighborhoods $V$ and $W$ of $\mu$ and $\nu$, respectively, such that $VW \subseteq U$. 

We begin by noting that if $\mu \neq 0$, then $B(\mu, d(\mu,0)) = \{\mu\}$, and therefore $\{\mu\}$ is an open set. It follows that if $\mu, \nu \neq 0$, then we may take $V = \{\mu\}$ and $W = \{\nu\}$ above. Next suppose that $\mu =0 = \mu\nu$ but $\nu \neq 0$. Write $\nu = pq^{-1}$ for some $p,q \in \path (E)$, let $W = \{\nu\}$, and let $V = B(0, 1/(|p|+n))$. Then for all $\rho = tz^{-1} \in V \setminus \{0\}$ ($t,z \in \path (E)$) we have $$d(\rho, 0) =  d(tz^{-1},0) = \frac{1}{\mathrm{min}\{|t|,|z|\} + 1} <  \frac{1}{|p|+n},$$ and hence $\mathrm{min}\{|t|,|z|\} > |p| +n-1$. It follows that either $z^{-1}pq^{-1} = 0$ or $(z^{-1}pq^{-1})^{-1} = qp^{-1}z \in \path (E)$ and $|(z^{-1}pq^{-1})^{-1}| > n -1 $. Thus, either $\rho\nu =0$ or, otherwise, $$d(\rho \nu,0) = d(tz^{-1}pq^{-1} ,0) = \frac{1}{\mathrm{min}\{|t|,|(z^{-1}pq^{-1})^{-1}|\} + 1} <  \frac{1}{n},$$ from which we see that $\rho \nu  \in B(0, 1/n) = B(\mu\nu, 1/n) = U$, and hence $VW \subseteq U$. A similar argument shows that if $\nu =0 = \mu\nu$ but $\mu \neq 0$, then we can find open neighborhoods $V$ and $W$ of $\mu$ and $\nu$, respectively, such that $VW \subseteq U$.

Finally, suppose that $\mu = \nu  = 0$, and let $V=W =B(0, 1/n)$.  Then for all  $\theta = pq^{-1} \in B(0, 1/n) (=V)$ and $\rho = tz^{-1} \in B(0, 1/n) (=W)$, with $p,q, t,z \in \path (E)$, either $\theta\rho = 0$, or $$d(\theta\rho,0) = d(pq^{-1}tz^{-1} ,0) \leq \frac{1}{\mathrm{min}\{|p|,|z|\} + 1} <  \frac{1}{n}.$$ Therefore $\theta\rho \in B(0, 1/n) = U$, as desired.
\end{proof}

\begin{corollary} \label{empty-cl-eg}
Let $E$ be a finite graph having at least one cycle, and suppose that $G(E)$ is a subsemigroup of a Hausdorff topological semigroup $S$. If $G(E)$ inherits from $S$ the topology induced by the metric $d$ from Proposition~\ref{non-disc-top}, then $\overline{G(E)} = G(E)$. 
\end{corollary}

\begin{proof}
Assume, to the contrary, that there exists $\mu \in \overline{G(E)}\setminus G(E)$. Since our topology is Hausdorff, there are open neighborhoods $U'$
and $V'$ (in $S$) of $0$ and $\mu$, respectively, such that $U'\cap V'=\emptyset$. Let $m\in \N$ be such that  $$B_{G(E)}(0, 1/m) : = \{\nu \in G(E) : d(0,\nu) < 1/m\} \subseteq U' \cap G(E).$$ Since $0 \cdot \mu=0 = \mu \cdot 0$ (as discussed in Subsection~\ref{top-sect}), by the continuity of multiplication, there are open subneighborhoods $U \subseteq U'$ and $V \subseteq V'$, such that $$(U \cap G(E))\cdot (V\cap G(E)), (V \cap G(E))\cdot (U\cap G(E))\subseteq B_{G(E)}(0, 1/m)$$ (and $B_{G(E)}(0, 1/m)\cap V=\emptyset$). Let $n\in \N$ be such that $n\geq m$ and $B_{G(E)}(0, 1/n)\subseteq U \cap G(E)$. Since $\mu$ is a limit point of $G(E)$ and $E$ is finite, it follows that $V\cap G(E)$ is infinite, and there exists $xy^{-1}\in V\cap G(E)$ such that $|x|\geq n$ or $|y|\geq n$. We may assume that $|y|\geq n$, since the case where $|x| \geq n$ leads to an analogous argument. Then $xy^{-1}=xy^{-1}yy^{-1}$. But, $$d(0, yy^{-1})=\frac{1}{|y|+1}< \frac{1}{n},$$ and hence $yy^{-1}\in B_{G(E)}(0, 1/n) \subseteq U\cap G(E)$. Therefore $$xy^{-1}=(xy^{-1})(yy^{-1})\in (V\cap G(E))\cdot (U\cap G(E))\subseteq B_{G(E)}(0, 1/m),$$ contradicting $B_{G(E)}(0, 1/m)\cap V=\emptyset$. Thus $ \overline{G(E)}= G(E)$.
\end{proof}

\section{Local Compactness}

A Hausdorff space $X$ is \emph{locally compact} if for every $x\in X$ there exists an open set $U$ such that $x\in U$ and $\overline{U}$ is compact. The main result of this section is that for finite graphs $E$, the discrete topology is the only possible locally compact Hausdorff semigroup topology on $G(E)$.

We begin with two short lemmas. The first is a well-known fact, but we provide a proof because we did not find a suitable reference.

\begin{lemma} \label{loc-comp-lemma}
Every countable locally compact Hausdorff space is metrizable.
\end{lemma}

\begin{proof}
Let $X$ be a locally compact Hausdorff space. According to~\cite[Theorem 3.3.1]{Engelking}, $X$ is completely regular. (A topological space $X$ is \emph{completely regular} or \emph{Tychonoff} if it is $T_1$, and given any closed subset $Y \subseteq X$ and any point $x \in X\setminus Y$, there is a continuous function $f : X \to [0,1] \subseteq \R$ such that $f(x) = 1$ and $f(Y)=\{0\}$.) Moreover, letting $w(X)$ denote the weight of $X$ (that is, the minimal cardinality of a base of $X$), \cite[Corollary 3.3.6]{Engelking} says that $w(X) \leq |X|$. Now, by Tychonoff's theorem~\cite[Theorem 2.3.23]{Engelking}, $X$ is homeomorphic to a subspace of $[0,1]^{w(X)}$. Hence, if $X$ is countable, then it is homeomorphic to a subspace of $[0,1]^{\N}$, which is metrizable. It follows that every countable locally compact Hausdorff space is metrizable.
\end{proof}

\begin{lemma}\label{pumping}
Let $E$ be a finite graph, and suppose that $\, \{x_n : n\in \N\}$ is an infinite subset of $\, \path (E)$. Then there exist an infinite subset $I$ of $\, \N$ and $\mu \in G(E)$ such that $\mu x_n \in \path (E)$ and $\, |\mu x_n|>|x_n|$ for all $n\in I$. 
\end{lemma}

\begin{proof}
Since $E$ is finite, there are $p, t \in \path (E)$, where $p$ is a cycle that is not a vertex, and an infinite subset $I \subseteq \N$, such that for all $n \in I$ we have $x_n = tpu_n$, for some $u_n \in \path (E)$. Letting $\mu = tpt^{-1}$, we see that $$|\mu x_n| = |tpt^{-1}tpu_n| = |tppu_n| > |tpu_n| = |x_n|$$ for all $n \in I$.
\end{proof}

\begin{theorem} \label{loc-comp}
If $E$ is a finite graph, then the only locally compact Hausdorff semigroup topology on $G(E)$ is the discrete topology. 
\end{theorem}

\begin{proof}
Seeking a contradiction, suppose that $G(E)$ has a locally compact Hausdorff semigroup topology which is not discrete. Since $E$ is finite, $G(E)$ is countable, and hence there is a metric $d$ that induces this topology on $G(E)$, by Lemma~\ref{loc-comp-lemma}. Furthermore, by Theorem~\ref{discrete}, $G(E)\setminus \{0\}$ must be discrete, and thus $0$ is the unique limit 
point in $G(E)$. Since $G(E)$ is locally compact, there exists $N\in \mathbb{N}$ such that $\overline{B(0, 1/n)}$ is compact for all $n\geq N$. Thus $\overline{B(0, 1/n)}\setminus B(0, 1/(n+1))$ is compact for all $n\geq N$, since it is a closed subset of a compact set, and so it is finite, being a  subset of a discrete space. Therefore if $X$ is any infinite subset of $B(0, 1/N)\setminus\{0\}$ and the elements of $X$ are arbitrarily enumerated as $\{x_0y^{-1}_0, x_1y^{-1}_1, \ldots\}$ ($x_n, y_n \in \path (E)$), then the sequence $(x_ny^{-1}_n)_{n\in\N}$ converges to $0$. 

There are three cases to consider.\vspace{\baselineskip}

\noindent\emph{Case 1: No sequence of elements in $\, \path(E)$ or $\, \path(E)^{-1}$ converges to $\, 0$.} Suppose that  there exists $y\in \path(E)$ such that the set $$X_y=\{x\in \path(E):0<d(xy^{-1}, 0)<1/N\}$$ is infinite, say $X_y=\{x_0, x_1, \ldots\}$. But then $x_ny^{-1}\to 0$ as $n\to \infty$, and so $x_n=x_ny^{-1}y\to 0$ as $n\to \infty$, contradicting the assumption of this case. Therefore $X_y$ is finite for all $y\in \path(E)$. Similarly, if there is $x\in \path(E)$ such that the set $$Y_x=\{y\in \path(E):0<d(xy^{-1}, 0)<1/N\}$$
is infinite, say $Y_x=\{y_0, y_1, \ldots\}$, then $xy_n^{-1}\to 0$ as $n\to \infty$, and so $y_n^{-1}=x^{-1}xy_n^{-1}\to 0$ as $n\to \infty$, contradicting the 
assumption of this case. Therefore $Y_x$ is finite for all $y\in \path(E)$.

Since $B(0, 1/N)\setminus \{0\}$ is infinite,  $X_{y}\not=\emptyset$ for infinitely many $y \in \path(E)$, and we denote these by $y_0, y_1, \ldots$. For every $n\in \N$, let $x_n\in X_{y_n}$ be of maximal length. Then $\{x_n y_n^{-1}: n \in \N\}$ is an infinite subset of $B(0, 1/N)\setminus\{0\}$, and hence $x_ny_n^{-1}\to 0$ as $n\to \infty$. Since, for each $x_n$, $Y_{x_n}$ is finite, $\{x_n : n \in \N\}$ must be infinite. Therefore, by Lemma \ref{pumping} there exists an infinite subset $I$ of $\N$ and $\mu \in G(E)$ such that $|\mu x_n|>|x_n|$ for all $n\in I$. Hence $\mu x_n\not\in X_{y_n}$, by the choice of $x_n$, and so $d(\mu x_ny_n^{-1}, 0)\geq 1/N$ for all $n\in I$. It follows that the sequence $(\mu x_ny_n^{-1})_{n\in I}$ does not converge to $0$ whereas $(x_ny_n^{-1})_{n\in I}$ does, contradicting the continuity of the multiplication in $G(E)$.

\vspace{\baselineskip}

\noindent\emph{Case 2: There exists a  sequence $\, (x_n)_{n\in \N}$ of elements in $\, \path(E)$ that converges to $\, 0$.} Since $E$ is finite, we may assume that for all $n,m \in \N$ we have $r(x_n)=r(x_m)$, $d(x_n, 0)<1/N$, and $|x_n|<|x_{n+1}|$, upon passing to a subsequence if necessary. In particular, $|x_n| \geq n$ for all $n \in \N$. We start by showing that $0$ is not a limit point of $\{x_n^{-1}:n\in \N\}$. Supposing that $0$ is a limit point, there exists a subsequence $(x_{k(n)})_{n\in \N}$ of $(x_n)_{n\in \N}$ such that $x_{k(n)}^{-1}\to 0$  as $n\to \infty$. But then, by the continuity of multiplication, $r(x_{k(n)})=x_{k(n)}^{-1}x_{k(n)}\to 0$, a contradiction.

Since $\overline{B(0, 1/n)}\setminus B(0, 1/(n+1))$ is finite for all $n\geq N$, it follows that the set $$\{|x|+|y|:xy^{-1}\in G(E),\ 1/p\leq d(xy^{-1}, 0)<1/N\}$$ 
is finite for all $p\geq N$. We denote the maximum of this set by $l(p)$. 
Let $p>N$ be arbitrary, and let $n(p)\in\N$ be such that $n(p)>l(p)$ and $x_{n(p)}\in B(0, 1/p)$. Since $0$ is not a limit point of $\{x_n^{-1}:n\in \N\}$, it follows that $(x_{n(p)}x_m^{-1})_{m\in \N}$ does not converge to $0$ as $m\to \infty$. (Note that $x_{n(p)}x_m^{-1} \neq 0$ since $r(x_n)=r(x_m)$ for all $n,m \in \N$.) Since $|x_n| < |x_{n+1}|$ for all $n\in \N$, the set $\{x_{n(p)}x_m^{-1} : m \in \N\}$ is infinite, and so there are infinitely many $m$ satisfying $d(x_{n(p)}x_m^{-1}, 0) \geq 1/N$. Fix $m(p)\in \N$ such that  $d(x_{n(p)}x_{m(p)}^{-1}, 0)\geq 1/N$, and write $x_{m(p)}^{-1}=e_1(p)^{-1}\cdots e_{k(p)}(p)^{-1}$, where $e_1(p),\ldots, e_{k(p)}(p)\in E^{1}$, Also let $j(p)\leq k(p)$ 
(possibly $j(p)=1)$ be such that $$d(x_{n(p)}e_1(p)^{-1}\cdots e_{j(p)-1}(p)^{-1}, 0)<\frac{1}{N}\qquad\text{ and } \qquad d(x_{n(p)}e_1(p)^{-1}\cdots e_{j(p)}(p)^{-1}, 0)\geq \frac{1}{N}.$$ But  then $$|x_{n(p)}| + |e_{j(p)-1}(p) \cdots e_1(p)|\geq |x_{n(p)}|\geq n(p)>l(p),$$ and so $$d(x_{n(p)}e_1(p)^{-1}\cdots e_{j(p)-1}(p)^{-1}, 0)<\frac{1}{p}.$$ It follows that  $x_{n(p)}e_1(p)^{-1}\cdots e_{j(p)-1}(p)^{-1}\to 0$ as $p\to \infty$. But, since $E$ is finite, some edge $e$ occurs infinitely many times as $e_{j(p)}(p)$ in the above construction. Thus as $p \to \infty$, $x_{n(p)}e_1(p)^{-1}\cdots e_{j(p)-1}(p)^{-1}e^{-1}$ does not converge to $0$, since only finitely many of the terms of this sequence are in $B(0, 1/N)$, which contradicts the continuity of multiplication.

\vspace{\baselineskip}

\noindent\emph{Case 3: There exists a  sequence $\, (x_n^{-1})_{n\in \N}$ of elements in $\, \path(E)^{-1}$ that converges to $\, 0$.} This can be handled analogously to Case 2.
\end{proof}

After an earlier version of this paper was circulated, Bardyla/Gutik~\cite[Proposition 3.4]{BG} proved that the conclusion of Theorem~\ref{loc-comp} also holds for graphs $E$ consisting of one vertex and infinitely many loops (i.e., infinitely-generated polycyclic monoids). However, the question of whether the theorem can be generalized to all graphs $E$ remains open.

\section{Idempotents}

The goal of this section is to characterize the inverse subsemigroups $S$ of $G(E)$ such that $\mu\nu \neq 0$ for all $\mu, \nu \in S \setminus \{0\}$. This characterization will be useful in subsequent sections. 

Recall that an element $\mu$ of a semigroup $S$ is an \emph{idempotent} if $\mu\mu = \mu$.

\begin{lemma} \label{idempt_lemma}
The following hold for any graph $E$.
\begin{enumerate}
\item[$(1)$] Every nonzero idempotent of $G(E)$ is of the form $xx^{-1}$ for some $x \in \path(E)$. 
\item[$(2)$] If $\mu, \nu \in G(E)$ are two idempotents, then $\mu\nu \in \{0, \mu, \nu\}$.
\end{enumerate}
\end{lemma}

\begin{proof}
(1) It is a standard fact that if $S$ is an inverse semigroup and $\mu \in S$ is an idempotent, then $\mu = \mu^{-1}$. (For, $\mu\mu\mu = \mu$ implies that $\mu = \mu^{-1}$, by the uniqueness of inverses.) Applying this to the inverse semigroup $G(E)$, suppose that $xy^{-1} \in G(E)$ is an idempotent ($x, y \in \path(E)$). Then $xy^{-1} = (xy^{-1})^{-1} =yx^{-1}$, from which the desired statement follows.

\vspace{\baselineskip}

(2) We may assume that $\mu \neq 0 \neq \nu$, since otherwise the claim is clear. By (1), we can write $\mu = xx^{-1}$ and $\nu = yy^{-1}$ for some $x, y \in \path(E)$. If $\mu\nu \neq 0$, then either $x=yz$ or $y=xz$ for some $z \in \path(E)$. In the first case, $$\mu\nu = xx^{-1}yy^{-1} = yzz^{-1}y^{-1}yy^{-1} = yzz^{-1}y^{-1} = xx^{-1} = \mu.$$ In the second case, $$\mu\nu = xx^{-1}yy^{-1} = xx^{-1}xzz^{-1}x^{-1} = xzz^{-1}x^{-1} = yy^{-1} = \nu.$$ Thus, $\mu\nu \in \{0, \mu, \nu\}$ for all idempotents $\mu$ and $\nu$.
\end{proof}

\begin{lemma} \label{no_zero_div_lemma}
Let $E$ be a graph, and suppose that $S$ is an inverse subsemigroup of $G(E)$ such that $\mu\nu \neq 0$ for all $\mu, \nu \in S \setminus \{0\}$. Then the following hold.
\begin{enumerate}
\item[$(1)$] Let $x \in \path(E)$, and set $S' = S \cap xG(E)x^{-1}$. Then $S'$ and $x^{-1}S'x$ are inverse subsemigroups of $G(E)$ satisfying the above hypothesis on $S$, and $f(\mu)=x^{-1}\mu x$ defines an isomorphism $f : S' \to x^{-1}S'x$.
\item[$(2)$] Suppose that $S \cap (\path(E)\setminus E^0)\neq \emptyset$, and let $x \in S \cap (\path(E)\setminus E^0)$ be such that $\, |x|$ is minimal. Then all nonzero elements of $S$ are of the form $x^nyy^{-1}x^{-m}$, where $n, m \in \N$, and $y \in \path (E)$ satisfies $x=yp$ for some $p \in \path(E)$.
\end{enumerate}
\end{lemma}

\begin{proof}
(1) We note that $xG(E)x^{-1}$ is an inverse semigroup, since for all $\mu, \nu \in G(E)$, $x\mu x^{-1}x\nu x^{-1} = x\mu\nu x^{-1}$ and $(x\mu x^{-1})^{-1} = x \mu^{-1}x^{-1}$. As an intersection of inverse semigroups, $S'$ must be one as well. Since $S'$ is a subsemigroup of $S$, clearly it has no zero-divisors.

Suppose that $yz^{-1}, uv^{-1} \in x^{-1}S'x \setminus \{0\}$, for some $u,v,y,z \in \path (E)$. Then $xyz^{-1}x^{-1}, \linebreak[0] xuv^{-1}x^{-1} \in S' \setminus \{0\}$, and since $$xyz^{-1}uv^{-1}x^{-1} = xyz^{-1}x^{-1}xuv^{-1}x^{-1} \in S',$$ it follows that $yz^{-1}uv^{-1} \in x^{-1}S'x$. Also, since $xzy^{-1}x^{-1} \in S'$, we have $zy^{-1} \in x^{-1}S'x$, showing that $x^{-1}S'x$ is an inverse semigroup.

To show that $S'$ and $x^{-1}S'x$ are isomorphic, let $f : S' \rightarrow x^{-1}S'x $ be as in the statement. This map is clearly a bijection. Letting $\mu, \nu \in S'$ be any elements, we can write $\mu = xyz^{-1}x^{-1}$ and $\nu = xuv^{-1}x^{-1}$ for some $u,v,y,z \in \path (E) \cup \{0\}$. Then $$f(\mu\nu) = f(xyz^{-1}uv^{-1}x^{-1}) = yz^{-1}uv^{-1} = f(\mu)f(\nu),$$ and hence $f$ is an isomorphism.

\vspace{\baselineskip}

(2) Let $\mu \in S\setminus \{0\}$ be any element, and write $\mu = uv^{-1}$ ($u,v \in \path (E)$). Let $n,m \in \N$ be maximal such that $x^{-n}u, (v^{-1}x^m)^{-1} \in \path (E)$. Since by hypothesis, $x^{-1}u, v^{-1}x \neq 0$, it follows that $u = x^ny$ and $v = x^mz$ for some $y,z \in \path (E)$, where $|y|,|z| < |x|$. Since $x^{-n}\mu x^m = yz^{-1} \in S$, by our choice of $x$ and the fact that $S$ is closed under inverses, either both $y, z \in E^0$, or both $y, z \notin E^0$. In the first case, $\mu = x^nx^{-m}$, giving $\mu$ the desired form. Let us therefore assume that $y, z \notin E^0$. Then $x^{-1}yz^{-1}x \neq 0$ implies that $x = yp=zq$ for some $p,q \in \path (E)\setminus E^0$, since $|y|,|z| < |x|$. Using the fact that $S$ is an inverse semigroup, and replacing $yz^{-1}$ by $zy^{-1}$, if necessary, we may assume that $|p| \leq |q|$. Now, $$0 \neq x^{-1}yz^{-1}x = p^{-1}y^{-1}yz^{-1}zq = p^{-1}q$$ implies that $q=pt$ for some $t \in \path (E)$, where $|t| < |q| < |x|$. But, then $p^{-1}q = p^{-1}pt = t \in S$ implies that $t \in E^0$, by the minimality of $|x|$. Hence $p=q$, and therefore also $y=z$, from which we obtain $\mu = uv^{-1} =  x^nyz^{-1}x^{-m} = x^nyy^{-1}x^{-m}$, as required.
\end{proof}

\begin{theorem} \label{no_zero_div_description}
Let $E$ be a graph, and suppose that $S$ is an inverse subsemigroup of $G(E)$ such that $\mu\nu \neq 0$ for all $\mu, \nu \in S \setminus \{0\}$. Then there exists an element $\mu \in S$ such that $S$ is generated as a semigroup by $\mu$ and the idempotents of $S$.
\end{theorem}

\begin{proof}
Let us assume that $S$ does not consist entirely of idempotents, since otherwise there is nothing to prove. Also, we may assume that $S \cap (\path(E)\setminus E^0)= \emptyset$, since otherwise the desired conclusion follows from Lemma~\ref{no_zero_div_lemma}(2).

Let $\mu = yx^{-1} \in S$ ($x, y \in \path (E)$) be a non-idempotent such that $|x|$ is minimal and $|y|$ is minimal for the chosen $x$. Since $S$ is an inverse semigroup, $xy^{-1} \in S$, and hence $|x| \leq |y|$. Therefore $0 \neq yx^{-1}yx^{-1}$ implies that $y = xp$ for some $p \in \path (E) \setminus E^0$. Hence $\mu = xpx^{-1}$, where $x \notin E^0$, since $S \cap (\path(E)\setminus E^0)= \emptyset$.

Now, let $vw^{-1} \in S$ be any non-idempotent ($v, w \in \path (E)$). Again, by assumption, $v, w \notin E^0$, and since $S$ is an inverse semigroup, our choice of $x$ implies that $|x| \leq |v|,|w|$. Thus, from $0 \neq xpx^{-1}vw^{-1}xp^{-1}x^{-1}$ we see that $vw^{-1} \in S': = S \cap xG(E)x^{-1}$. By Lemma~\ref{no_zero_div_lemma}(1), $x^{-1}S'x$ is an inverse semigroup satisfying the hypothesis on $S$. Since $p \in x^{-1}S'x \cap (\path(E)\setminus E^0)$ and $|p|$ is minimal (by our choice of $y$), Lemma~\ref{no_zero_div_lemma}(2) implies that $x^{-1}S'x$ is generated by $p$ and some set of idempotents. By Lemma~\ref{no_zero_div_lemma}(1), $x^{-1}S'x$ is isomorphic to $S'$, via an isomorphism that sends $p$ to $\mu$, and therefore $S'$ is generated by $\mu$ and some set of idempotents. Since all the non-idempotents of $S$ are elements of $S'$, the desired conclusion follows.
\end{proof}

It is not hard to see that if the $\mu = xpx^{-1}$ above is not an idempotent, then $S$ is generated by $\mu$ and idempotents of the form $uu^{-1}$ ($u\in \path (E)$), where $uq=xp$ for some $q\in \path (E)$.

\section{Closures}

Thus far we have considered all possible Hausdorff topologies on $G(E)$ which made multiplication continuous. Now we restrict our attention to topologies with respect to which inversion is also continuous. More specifically, using the theory developed above we shall describe the complement of $G(E)$ in the closure $\overline{G(E)}$ of $G(E)$ in any topological inverse semigroup that contains it. We begin by recalling a couple of basic facts about topological inverse semigroups. The second, which will be crucial for us, relies heavily on the assumption that inversion (and not just multiplication) is continuous.

\begin{proposition}[Proposition II.2 in~\cite{ES}] \label{inverse_subsemi}
Let $S$ be a topological inverse semigroup and $T$ an inverse subsemigroup of $S$. Then $T$ and $\overline{T}$ are topological inverse subsemigroups of $S$.
\end{proposition}

\begin{proposition}[Proposition II.3 in~\cite{ES}] \label{dense_idempt}
Let $S$ be a topological inverse semigroup and $T$ a dense inverse subsemigroup of $S$. Also, let $I$ denote the set of all idempotents of $S$. Then $I = \overline{I\cap T}$.
\end{proposition}

The next result is a generalization of \cite[Proposition III.1]{ES}, which says that, letting $\overline{P_1}$ denote the closure of the bicyclic monoid $P_1$ in a Hausdorff topological semigroup, $\overline{P_1} \setminus P_1$ is a group.

\begin{proposition} \label{idempt_in_T}
Let $E$ be a graph, and suppose that $G(E)$ is a subsemigroup of a Hausdorff topological inverse semigroup. Set $T = \overline{G(E)}\setminus G(E)$, and let $I$ denote the subset of all idempotents of $\overline{G(E)}$. Then the following hold.
\begin{enumerate}
\item[$(1)$] For all $\rho \in T$ there are idempotents $\mu, \nu \in I \cap T$ such that $\rho \in \mu T\nu$.
\item[$(2)$] For all $\mu, \nu \in I \cap T$, if $\mu \neq \nu$, then $\mu\nu = 0$.
\item[$(3)$] For all $\mu \in I \cap T$ the set $\mu T \mu \setminus \{0\}$ is a group with identity $\mu$.
\item[$(4)$] For all $\mu \in I\cap T$, $\mu G(E) \mu$ is dense in $\mu \overline{G(E)} \mu = \mu T\mu \cup \{0\}$.
\end{enumerate}
\end{proposition}

\begin{proof}
(1) Let $\mu = \rho\rho^{-1}$ and $\nu = \rho^{-1}\rho$. Then $\mu, \nu \in I$, and $\rho = \rho\rho^{-1}\rho\rho^{-1}\rho = \mu\rho\nu \in \mu T\nu$. Also, since $\rho\rho^{-1}\rho = \rho \neq 0$ and $\rho^{-1}\rho\rho^{-1} = \rho^{-1} \neq 0$, Theorem~\ref{complement_ideal}(2) implies that $\mu, \nu \in T$.

\vspace{\baselineskip}

(2) Suppose that $\mu, \nu \in I \cap T$ are such that $\mu\nu \neq 0$. Suppose further that $\mu \neq \mu\nu$ (so in particular, $\mu \neq \nu$). Then there exist open neighborhoods $U$ of $\mu$, $V$ of $\nu$, and $W$ of $\mu\nu$ such that $0 \notin U\cup V\cup W$, $U \cap W = \emptyset$, and $UV \subseteq W$ (by continuity of multiplication and the fact that the topology is Hausdorff). By Proposition~\ref{dense_idempt}, $I=\overline{I\cap G(E)}$, from which it follows that $U \cap (I \cap G(E))$ and $V \cap (I \cap G(E))$ are infinite. By Lemma~\ref{idempt_lemma}(1), every idempotent of $G(E)$ is of the form $xx^{-1}$ for some $x \in \path (E)$. Hence, there are $xx^{-1} \in U \cap (I \cap G(E))$ and $yy^{-1} \in V \cap (I \cap G(E))$, for some $x, y \in \path(E)$. Since $0 \notin W$, we have $xx^{-1}yy^{-1} \in \{xx^{-1}, yy^{-1}\}$, by Lemma~\ref{idempt_lemma}(2). Since $U \cap (I \cap G(E))$ is infinite, and $xx^{-1}yy^{-1} = yy^{-1}$ for only finitely many values of $x$ (namely, $x$ satisfying $xp=y$ for some $p \in \path (E)$), we may choose $x$ such that $xx^{-1}yy^{-1} = xx^{-1}$. This, however, contradicts $U \cap W = \emptyset$. Thus, $\mu = \mu\nu$, and by a similar argument, $\mu\nu = \nu$. Therefore, $\mu = \nu$.

\vspace{\baselineskip}

(3) Let $\rho, \tau \in \mu T \mu \setminus \{0\}$ be any elements. Then $\rho\rho^{-1}, \rho^{-1}\rho \in I\cap T$, and since $\mu\rho\rho^{-1} = \rho\rho^{-1} \neq 0$ and $\rho^{-1}\rho \mu = \rho^{-1}\rho \neq 0$, (2) implies that $\mu = \rho\rho^{-1} = \rho^{-1}\rho$. By similar reasoning $\mu = \tau\tau^{-1}$, and hence $\mu = \mu\mu = \rho^{-1}\rho \tau\tau^{-1}$, implying that $\rho \tau \neq 0$. Thus, by Theorem~\ref{complement_ideal}(2), $\rho\tau = \mu\rho\tau\mu \in \mu T \mu \setminus \{0\}$. Also, $$\rho^{-1} = (\mu\rho\mu)^{-1} = \mu^{-1} \rho^{-1} \mu^{-1} = \mu \rho^{-1} \mu \in \mu T \mu \setminus \{0\}.$$ It follows that $\mu T \mu \setminus \{0\}$ is a group with identity $\mu$.

\vspace{\baselineskip}

(4) By Theorem~\ref{complement_ideal}(2), $T \cup \{0\}$ is an ideal in $\overline{G(E)}$, and hence $\mu G(E) \mu \subseteq T \cup \{0\}$. Since $\overline{G(E)} = T \cup G(E)$, it follows that $\mu \overline{G(E)} \mu = \mu T\mu \cup \{0\}$. Also, $\mu G(E)\mu$ is dense in $\overline{\mu G(E)\mu}$, which contains $\mu \overline{G(E)}\mu$, by the continuity of multiplication, and hence $\mu G(E)\mu$ is dense in $\mu \overline{G(E)}\mu$.
\end{proof}

\begin{lemma} \label{id}
Let $E$ be a graph, suppose that $G(E)$ is a subsemigroup of a Hausdorff topological semigroup, and let $\mu \in \overline{G(E)}$ be an idempotent. Then there is a vertex $v \in E^0$ such that $v\mu = \mu = \mu v$.
\end{lemma}

\begin{proof}
First, suppose that $v, w \in E^0$ are such that $v \mu = \mu = \mu w$. Then $\mu = \mu\mu = \mu w v \mu$ implies that $v=w$ (since $0\cdot \overline{G(E)} = \{0\} = \overline{G(E)} \cdot 0$), and hence $v\mu = \mu = \mu v$. In particular, the statement clearly holds for idempotents $\mu$ in $G(E)$. Let us therefore assume that $\mu \in \overline{G(E)} \setminus G(E)$ and find $v, w \in E^0$ such that $v \mu = \mu = \mu w$. 

Let $U$ be an open neighborhood of $\mu$ such that $0 \notin U$. Since $\mu \mu = \mu$, by the continuity of multiplication we can find an open neighborhood $V$ of $\mu$ such that $V\mu \subseteq U$. Since $\mu$ is a limit point of $G(E)$, we can find some $\rho = xy^{-1} \in V \cap G(E)$ ($x, y \in \path (E)$). Since $0 \notin U$, we have $\rho\mu \neq 0$, and hence $s(y)\mu \neq 0$. By a similar argument, there must be some $v \in E^0$ such that $\mu v \neq 0$.

Now, suppose that $v\mu = \nu \neq 0$ for some $v \in E^0$ and $\nu \in \overline{G(E)} \setminus \{\mu\}$. Then we can find open neighborhoods $U$ and $V$ of $\mu$ and $\nu$, respectively, such that $vU \subseteq V$, $U\cap V = \emptyset$, and $0 \notin V$. Again, there must be some element $\rho \in U \cap G(E)$, and necessarily $v\rho \neq 0$. It follows that $\rho = v\rho \in U \cap V$; a contradiction. Hence if $v \mu \neq 0$, then $v\mu = \mu$, and similarly if $\mu v \neq 0$, then $\mu v = \mu$.
\end{proof}

The following is a generalization of \cite[Corollary III.3]{ES}, which says that, again letting $\overline{P_1}$ denote the closure of the bicyclic monoid $P_1$ in a Hausdorff topological inverse semigroup, the group $\overline{P_1} \setminus P_1$ contains a dense cyclic subgroup.

\begin{theorem} \label{comp-gp-thrm}
Let $E$ be a graph, and suppose that $G(E)$ is a subsemigroup of a Hausdorff topological inverse semigroup. Set $T = \overline{G(E)}\setminus G(E)$, and let $\mu \in T$ be an idempotent. Then either the group $\mu T \mu \setminus \{0\}$ is trivial, or it contains a dense cyclic subgroup.
\end{theorem}

\begin{proof}
Let $A = \{\nu \in G(E) : \mu\nu = \nu\mu \neq 0\}$. By Lemma~\ref{id}, we can find $v \in E^0$ such that $v\mu = \mu v = \mu$, which implies that $A \neq \emptyset$. We wish to show that $A$ is an inverse semigroup. Let $\nu, \gamma \in A$. Then $\mu\nu = \mu\nu\mu \in \mu T \mu \setminus \{0\}$, by Theorem~\ref{complement_ideal}(2), and since $\mu T \mu\setminus \{0\}$ is a group, we have $\mu\nu\gamma = (\mu\nu\mu)(\mu\gamma\mu) \in \mu T \mu\setminus \{0\}$. Therefore $\mu\nu\gamma \neq 0$, and since $\mu\nu\gamma = \nu\mu\gamma = \nu\gamma\mu$, this implies that $\nu\gamma \in A$. Also, for any $\nu  \in A$ we have $$\mu\nu^{-1}\mu = (\mu^{-1}\nu\mu^{-1})^{-1} = (\mu\nu\mu)^{-1} = (\mu\nu)^{-1} = (\nu\mu)^{-1}.$$ Thus $\mu\nu^{-1} = \nu^{-1}\mu \neq 0$, and therefore $\nu^{-1} \in A$, showing that $A$ is an inverse semigroup. Also, since $\nu\gamma \neq 0$ for all $\nu, \gamma \in A$, by Theorem~\ref{no_zero_div_description} there is an element $\tau \in A$ such that $A$ is generated by $\tau$ and the set $I$ of idempotents of $A$. Note that $I \neq \emptyset$, since as mentioned above, $A$ must contain a vertex.

Now, define $f : A \rightarrow \mu T \mu \setminus \{0\}$ by $\nu \mapsto \mu\nu \, (= \mu\nu\mu)$. Then $$f(\nu\gamma) = \mu\nu\gamma = \mu\nu\mu\gamma = f(\nu)f(\gamma)$$ for all $\nu, \gamma \in A$, and hence $f$ is a homomorphism. Since there is only one idempotent in any group, $f(I) = \{\mu\}$, and hence $f(A)$ is the subgroup of $\mu T \mu \setminus\{0\}$ generated by $f(\tau)$. Thus, either $f(A) = \{\mu\}$, or $f(A)$ is a cyclic subgroup of $\mu T \mu \setminus\{0\}$. Therefore, to prove the theorem we need only show that $f(A)$ is dense in $\mu T \mu \setminus \{0\}$.

Let $\rho \in \mu T \mu \setminus \{0\}$ be any element, and let $U$ be an open neighborhood of $\rho$. Since our topology is Hausdorff, we may assume that $0 \notin U$. Since $\mu\rho = \rho = \rho \mu$, by the continuity of multiplication, we can find an open neighborhood $V$ of $\rho$ such that $\mu V \mu \subseteq U$. Let $\delta \in V \cap G(E)$ be any element. Since $\mu\delta \neq 0 \neq \delta\mu$, by Theorem~\ref{complement_ideal}(2) we have  $\mu\delta, \delta\mu \in T \setminus \{0\}$. By Proposition~\ref{idempt_in_T}(1), there are idempotents $\theta, \eta \in T$ such that $\mu \delta \in \theta T \eta$. But, since $\mu(\mu\delta)\mu \neq 0$, by Proposition~\ref{idempt_in_T}(2), this can only happen if $\mu = \theta = \eta$. Therefore, $\mu\delta \in \mu T \mu \setminus \{0\}$, and similarly $\delta\mu \in \mu T \mu \setminus \{0\}$. Hence, $\mu \delta = \mu \delta \mu = \delta \mu$, from which it follows that $\delta \in A$, and therefore $\mu\delta \in f(A) \cap U$. Thus $f(A)$ is dense in $\mu T \mu \setminus \{0\}$.
\end{proof}

\section{Polycyclic Monoids}

Recall that if $E$ is a graph having only one vertex $v$ and $n$ edges (necessarily loops), for some integer $n \geq 1$, then $G(E)$ is known as a \emph{polycyclic monoid}, and we denote it by $P_n$.

We conclude this article with some observations on the possible sizes of the closures of $P_n$ inside larger topological semigroups.

\begin{proposition}
Let $n \geq 2$ be an integer, and suppose that $P_n$ is a subsemigroup of a Hausdorff topological semigroup. Then $\overline{P_n}\setminus P_n$ is either empty or infinite.
\end{proposition}

\begin{proof}
Suppose that $\overline{P_n}\setminus P_n \neq \emptyset$, and let $\mu \in \overline{P_n}\setminus P_n$. Letting $e_1, \dots, e_n$ be the generators of $P_n$ as an inverse semigroup with zero, we have $$P_n = \{0,1\} \cup \bigcup_{i=1}^n e_iP_n \cup \bigcup_{i=1}^n P_ne_i^{-1},$$ and hence $$\overline{P_n} = \overline{\{0,1\}} \cup \bigcup_{i=1}^n \overline{e_iP_n} \cup \bigcup_{i=1}^n \overline{P_ne_i^{-1}}.$$ Since $\overline{\{0,1\}} = \{0,1\}$, either $\mu \in \overline{e_iP_n}$ or $\mu \in \overline{P_ne_i^{-1}}$ for some $i$. Let us assume that $\mu \in \overline{e_1P_n}$, as the other cases can be handled analogously. We wish to show that $e_2^j \mu \neq e_2^k \mu$ for all distinct $j,k \in \N$. The desired result will then follow, since by Theorem~\ref{complement_ideal}(2), $e_2^j \mu \in (\overline{P_n} \setminus P_n)\cup \{0\}$ for all $j$. (It is easy to see that $\overline{P_n}$ is a topological semigroup.)

Suppose, on the contrary, that $e_2^j \mu = e_2^k \mu$ for some $j > k$, and hence that $\mu =e_2^{k-j}\mu$. Since $\mu \in \overline{e_1P_n}$, by the continuity of multiplication, it follows that $\mu =e_2^{k-j}\mu \in \overline{e_2^{k-j}e_1P_n}$. But $e_2^{k-j}e_1 = 0$, and hence $\mu = 0$ (since $0 \cdot \overline{P_n} = \{0\} = \overline{P_n} \cdot 0$), contradicting our choice of $\mu$. Thus, $e_2^j \mu \neq e_2^k \mu$ for all distinct $j,k \in \N$, as required.
\end{proof}

In contrast to the above result, $P_1$ can be embedded in a metrizable semigroup $S$, such that  $\overline{P_1}=S$ and $|S\setminus P_1| = 1$, as the next example shows.

\begin{example}
Let $S=P_1\cup\{\delta\}$, and extend the multiplication operation of $P_1$ to $S$ as follows. For all $\mu \in P_1 \setminus \{0\}$ let $\mu\delta=\delta \mu=\delta$, set $0\delta=\delta 0=0$, and let $\delta \delta=\delta$. Then $S$ is clearly an inverse semigroup (with $\delta^{-1} = \delta$). Define $d : S \times S \to \R$ by $d(\mu,\nu)=1$ if $(\mu,\nu)\in (P_1\times P_1)\cup \{(0,\delta), (\delta,0)\}$ with $\mu \neq \nu$, set $$d(\delta, e^ne^{-m}) = d(e^ne^{-m},\delta) =\frac{1}{\min\{n,m\}+1}$$ for all $n, m \in \N$, where $e$ is the generator of $P_1$ as an inverse semigroup with zero, and let $d(\mu,\nu)=0$ whenever $\mu = \nu$. It is easy to check that $d$ is a metric, and that $\overline{P_1}=S$.

It remains to show that the multiplication in $S$ is continuous with respect
to the topology induced by $d$. Letting $\mu,\nu\in S$ be any elements and $U$ an open neighborhood of $\mu\nu$, we wish to find open neighborhoods $V$ and $W$ of $\mu$ and $\nu$, respectively, such that $VW \subseteq U$. First suppose that $\mu\nu=0$. Then either $\mu=0$ or $\nu=0$. Let us assume that $\mu=0$, as the other case can be handled analogously. Then taking $V = \{0\}$ and $W=S$, we have $VW \subseteq U$. We may therefore assume that $\mu\nu\neq 0$, and hence that $\mu\neq 0 \neq \nu$. Now, view $P_1$ as a topological semigroup, using the topology constructed in Proposition~\ref{non-disc-top}. Then defining $F: S\setminus\{0\} \to P_1$ by $F(\tau)=\tau$ for $\tau \neq \delta$ and $F(\delta)=0$, gives a homeomorphism. It follows that there are open neighborhoods $V$ and $W$ of $\mu$ and $\nu$, respectively, such that $VW \subseteq U$, when $\mu,\nu,\mu\nu \neq 0$. Hence multiplication is continuous on all of $S$. 
\end{example}

Returning to $P_n$ with $n \geq 2$, we next construct a metrizable topological semigroup $S$ containing a dense copy of $P_2$, such that $|S\setminus P_2|=\aleph_0$ and $P_2$ is not discrete.

\begin{example}
Let $S$ be the monoid with zero element defined by the presentation:
$$\langle e, f, e^{-1}, f^{-1}, X : e^{-1}f=f^{-1}e=0,\ e^{-1}e=f^{-1}f=1,
\ eX=Xf^{-1}=X, \ e^{-1}X=Xf=0\rangle.$$ Let $\langle e,f \rangle$ denote the subsemigroup of $S$ generated by $\{e,f\}$, and set $$A=\{xe : x \in \langle e,f \rangle\} \cup\{1\} \ \text{ and } \ B= \{xf : x \in \langle e,f \rangle\} \cup\{1\}.$$ It is easy to show that every element of $S$ is of the form: $ue^mf^{-n}v^{-1}$ or $uXv^{-1}$, where $m,n\in\N$, $u\in B$, and $v\in A$. In particular, $S$ contains a copy $P_2$ of the polycyclic inverse monoid on two generators. Note that $$X^2=Xf^{-1}eX=0,\quad Xe=Xf^{-1}e=0,\quad f^{-1}X=f^{-1}eX=0.$$ So that every element of $S$ can be written in the form $ue^{m}f^{-n}v^{-1}$, we make the conventions $X=1\cdot e^{\infty}f^{-\infty}\cdot 1$ and $0=0\cdot e^{\infty}f^{-\infty}\cdot 0$.

Given any $\sigma \in S$, write $\sigma = u_\sigma e^{m_\sigma}f^{-n_\sigma}v_\sigma^{-1}$, with $m_\sigma, n_\sigma \in \N \cup \{\infty\}$, where $u_\sigma \in \{0,1\}$ or $u_\sigma=e^{a_0}f^{a_1}\cdots e^{a_{k(\sigma)-1}}f^{a_{k(\sigma)}},$ and $v_\sigma \in \{0,1\}$ or $v_\sigma=f^{b_0}e^{b_1}\cdots f^{b_{l(\sigma)-1}}e^{b_{l(\sigma)}},$ for some $k(\sigma), l(\sigma) \geq 1$, $a_0, b_0 \in \N$ and $a_1, \ldots, a_{k(\sigma)}, b_1, \ldots, b_{l(\sigma)} \in \N \setminus \{0\}$. Using this notation, define $\Xi : S \to \N \cup \{\infty\}$ by
\begin{equation*}
  \Xi(\sigma)=
  \begin{cases}
    0&\text{if }u_\sigma=1\text{ or }v_\sigma=1\\
    \infty&\text{if } \sigma = 0\\
    \min\{k(\sigma), l(\sigma) \}&\text{otherwise}.
  \end{cases}
\end{equation*}
Note that $\Xi(\sigma \tau)\geq \min\{\Xi(\sigma), \Xi(\tau)\}$ for all $\sigma, \tau\in S$. Using the same notation, define $\Delta:S\times S\to \R$ by 
\begin{equation*}
  \Delta(\sigma, \tau)=
  \begin{cases}
    0 & \text{if } (u_\sigma, v_\sigma) = (u_\tau, v_\tau)\\
   \displaystyle \frac{1}{1+\min\{\Xi(\sigma), \Xi(\tau)\}} & \text{if }  (u_\sigma, v_\sigma) \neq (u_\tau, v_\tau),
  \end{cases}
\end{equation*}
and define $\Phi:S\times S\to \R$ by 
\begin{equation*}
  \Phi(\sigma, \tau)=
  \begin{cases}
    0 & \text{if } (m_\sigma, n_\sigma) = (m_\tau, n_\tau)\\
   \displaystyle \frac{1}{1+\min\{m_\sigma, n_\sigma, m_\tau, n_\tau\}} & \text{if }  (m_\sigma, n_\sigma) \neq (m_\tau, n_\tau).
  \end{cases}
\end{equation*}
Finally, define $d:S\times S\to \R$ by $$d(\sigma, \tau) =\Delta(\sigma, \tau)+\Phi(\sigma, \tau).$$ It can be shown easily that $d$ is a metric, using the fact that $\Delta$ and $\Phi$ are symmetric and satisfy the triangle inequality.

Since every element of $S\setminus P_2$ is of the form $ue^{\infty}f^{-\infty}v^{-1}$ for some $u\in B$ and $v\in A$, $$d(ue^nf^{-n}v^{-1}, ue^{\infty}f^{-\infty}v^{-1})=\Phi(ue^nf^{-n}v^{-1}, ue^{\infty}f^{-\infty}v^{-1})=\frac{1}{1+n}$$ holds for all $n \in \N$, and so $P_2$ is dense in $S$. Also, if $\sigma=u_\sigma e^{m_\sigma}f^{-n_\sigma}v_\sigma^{-1}\in P_2\setminus \{0\}$ and $\tau \in S$ are any elements, then $m_\sigma, n_\sigma \in\N$ and $$d(\sigma, \tau)\geq \Phi(\sigma, \tau) \geq \frac{1}{1+\min\{m_\sigma, n_\sigma\}},$$ and hence $P_2\setminus \{0\}$ is discrete in $S$. Keeping $\sigma=u_\sigma e^{m_\sigma}f^{-n_\sigma}v_\sigma^{-1}\in P_2\setminus \{0\}$ as before, we also see that $$d(\sigma, 0) =  \frac{1}{1+\Xi(\sigma)} + \frac{1}{1+\min\{m_\sigma, n_\sigma\}},$$ which implies that $0$ is a limit point, since $\Xi(\sigma)$, $m_\sigma$, and $n_\sigma$ can be made arbitrarily large by choosing $\sigma$ appropriately. Also, $$d(\sigma, \tau) = \Delta (\sigma, \tau) \geq \frac{1}{1+\Xi(\sigma)}$$ for all distinct $\sigma, \tau \in S\setminus P_2$, and so $S\setminus P_2$ is discrete in $S\setminus P_2$. Finally, it is clear that $|S \setminus P_2| = \aleph_0$.

It remains to show that $S$ is a topological semigroup with respect to the topology induced by $d$. We shall do so by proving that for arbitrary $\sigma, \tau \in S$ and $n\in\N \setminus \{0\}$, there exists $m\in \N \setminus \{0\}$ such that $B(\sigma, 1/m) B(\tau, 1/m)\subseteq B(\sigma \tau, 1/n)$. There are several cases to consider; those not covered below follow by symmetry.

\vspace{\baselineskip}

\noindent\emph{Case 1: $\sigma=0$ and $\tau=0$.}
Set $m = 2n$, and let $\mu,\nu\in B(0, 1/m)$ be arbitrary. If $\mu\nu = 0$, then $\mu\nu \in B(0, 1/n) = B(\sigma \tau, 1/n)$. Let us therefore assume that $\mu\nu \neq 0$. Then $\mu\nu=(u_\mu e^{m_{\mu}}f^{-n_{\mu}}v_{\mu}^{-1}) (u_\nu e^{m_{\nu}}f^{-n_{\nu}}v_{\nu}^{-1})$, and hence either $\mu\nu=
u_\mu e^{m_{\mu}}we^{m_{\nu}}f^{-n_{\nu}}v_{\nu}^{-1}$  for some $w\in B$ or
$\mu\nu= u_\mu e^{m_{\mu}}f^{-n_{\mu}}w^{-1}f^{-n_{\nu}}v_{\nu}^{-1}$
for some $w\in A$. In either case,
\begin{eqnarray*}
  d(\mu\nu, 0)=\Delta(\mu\nu, 0) + \Phi(\mu\nu, 0)
  &\leq& \frac{1}{1+\Xi(\mu\nu)} + \frac{1}{1+\min\{m_{\mu\nu}, n_{\mu\nu}\}}\\
  &\leq& \frac{1}{1+\min\{\Xi(\mu),\Xi(\nu)\}} +  \frac{1}{1+\min\{m_\mu, n_{\mu}, m_{\nu}, n_{\nu}\}}\\
  &\leq& \Delta(\mu, 0)+\Delta(\nu, 0) + \Phi(\mu, 0)+\Phi(\nu, 0)\\
  &=& d(\mu, 0)+d(\nu, 0)<\frac{1}{m} + \frac{1}{m} = \frac{1}{n},
\end{eqnarray*}
and hence $\mu\nu \in B(0, 1/n) = B(\sigma \tau, 1/n)$.

\vspace{\baselineskip} 

\noindent\emph{Case 2: $\sigma=0$ and $\tau \in S\setminus \{0\}$.} As usual we write $\tau=u_\tau e^{m_\tau}f^{-n_\tau}v_\tau^{-1}$. If $\tau \in P_2\setminus \{0\}$, then, since $P_2\setminus\{0\}$ is discrete, we can choose $m > \Xi(\tau)+2n$ such that $B(\tau, 1/m)=\{\tau\}$. If $\tau \in S\setminus P_2$, we choose $m > \Xi(\tau)+2n$ so that $\nu=u_\tau e^{m_\nu}f^{-n_\nu}v_\tau^{-1}$ for all $\nu\in B(\tau, 1/m)$. (Note that for all $\nu \in S$, if $(u_\nu, v_\nu) \neq (u_\tau, v_\tau)$, then $\Delta (\nu, \tau) \geq  1/(1+\Xi(\tau))$.) 

Now, in either situation $$d(\mu, 0) = \Delta (\mu, 0) + \Phi (\mu, 0) \leq \frac{1}{1+\Xi (\mu)} +  \frac{1}{1+\min\{m_\mu, n_\mu\}}  < \frac{1}{m}$$ holds for all $\mu\in B(0, 1/m) \setminus \{0\}$, which implies that $\min\{n_\mu, m_\mu\} \geq m >2n$ and $\Xi (\mu) \geq m > \Xi(\tau)$. From the latter we also see that $|v_\mu| > |u_\tau|$. Thus, for all $\mu\in B(0, 1/m)$ and $\nu\in B(\tau, 1/m)$, either $\mu \nu=0$, or $v_{\mu}=u_{\tau} e^{m_{\nu}}w$ for some $w \in A$ and $\mu \nu=u_{\mu}e^{m_{\mu}}f^{-n_\mu}w^{-1}f^{-n_\nu}v_{\tau}^{-1}.$ In the latter case, $\Xi(\mu\nu)\geq \Xi(\mu) \geq m > 2n$, and $$\Phi(\mu\nu, 0) \leq  \frac{1}{1+\min\{m_{\mu\nu}, n_{\mu\nu}\}} \leq \frac{1}{1+\min\{m_\mu, n_\mu\}} < \frac{1}{2n}$$ (since $m_{\mu}, n_{\mu}> 2n$), from which it follows that $$d(\mu \nu, 0) = \Delta(\mu \nu, 0) + \Phi (\mu \nu, 0) < \frac{1}{2n} + \frac{1}{2n} = \frac{1}{n}.$$ Thus $\mu \nu\in B(0, 1/n) = B(\sigma \tau, 1/n)$ for all $\mu\in B(0, 1/m)$ and $\nu\in B(\tau, 1/m)$, as required.

\vspace{\baselineskip}

\noindent\emph{Case 3: $\sigma, \tau \in P_2\setminus\{0\}$.}
Since $P_2\setminus\{0\}$ is discrete, $\{\sigma\}$ and $\{\tau\}$ are open in
$S$, and hence we can find an $m \in \N \setminus \{0\}$ such that $\{\sigma\} = B(\sigma, 1/m)$ and $\{\tau\} = B(\tau, 1/m)$. Then $B(\sigma, 1/m) B(\tau, 1/m)\subseteq B(\sigma \tau, 1/n)$.

\vspace{\baselineskip}

\noindent \emph{Case 4: $\sigma \in P_2\setminus \{0\}$ and $\tau \in S\setminus P_2$.} 
Write $\sigma = u_\sigma e^{m_\sigma}f^{-n_\sigma}v_\sigma^{-1}$ and $\tau = u_\tau Xv_\tau^{-1}$. Let $m \in \N$ be such that $m> \max\{n, |v_\sigma|+n_\sigma\}$, $B(\sigma, 1/m)=\{\sigma\}$, and $\nu=u_\tau e^{m_{\nu}}f^{-n_{\nu}}v_\tau^{-1}$ for all $\nu\in B(\tau, 1/m)$. Note that as before, $m_{\nu}, n_{\nu} \geq m$ for all $\nu\in B(\tau, 1/m)$.

If $\sigma \tau\neq 0$, then $u_\tau=v_\sigma f^{n_\sigma}w$ for some $w\in \langle e,f \rangle \cup \{1\}$ (since $e^{-1}X = 0 = f^{-1}X$), and hence $\sigma \tau=u_\sigma e^{m_\sigma}wXv_\tau^{-1}$ and $$\sigma \nu = (u_\sigma e^{m_\sigma}f^{-n_\sigma}v_\sigma^{-1})(u_\tau e^{m_{\nu}}f^{-n_{\nu}}v_\tau^{-1}) = u_\sigma e^{m_\sigma}w e^{m_\nu}f^{-n_\nu} v_\tau^{-1}$$ for all $\nu \in B(\tau, 1/m)$. This implies that $$d(\sigma \nu, \sigma \tau) = \Delta (\sigma \nu, \sigma \tau) + \Phi(\sigma \nu, \sigma \tau)  =  0 +  \frac{1}{1+\min\{m_{\sigma \nu}, n_{\sigma \nu}\}} \leq \frac{1}{1+\min\{m_\nu, n_\nu\}} < \frac{1}{m}$$ for all $\nu \in B(\tau, 1/m)$, and hence $B(\sigma, 1/m) B(\tau, 1/m)\subseteq B(\sigma \tau, 1/m)\subseteq B(\sigma \tau, 1/n)$. 

If $\sigma \tau=0$, then either $f^{-n_\sigma}v_\sigma^{-1}u_\tau=0$ or $v_\sigma f^{n_\sigma}=u_\tau wf$ for some $w\in \langle e,f \rangle \cup \{1\}$. In the former case, $$\sigma \nu=(u_\sigma e^{m_\sigma} f^{-n_\sigma}v_\sigma^{-1})(u_\tau e^{m_{\nu}}f^{-n_{\nu}}v_\tau^{-1})=0$$  for all $\nu \in B(\tau, 1/m)$. In the latter case, $$\sigma \nu = (u_\sigma e^{m_\sigma} f^{-1}w^{-1}u_\tau^{-1})(u_\tau e^{m_{\nu}}f^{-n_{\nu}}v_\tau^{-1}) = u_\sigma e^{m_\sigma}f^{-1}w^{-1}e^{m_{\nu}} f^{-n_{\nu}}v_\tau^{-1}$$ for all $\nu \in B(\tau, 1/m)$, and since $m_{\nu} \geq m>|v_\sigma|+n_\sigma \geq |w|+1$, we see that $f^{-1}w^{-1}e^{m_{\nu}}=0$. Either way, $\sigma \nu=0 \in B(0, 1/n) = B(\sigma \tau, 1/n)$ for all $\nu \in B(\tau, 1/m)$, and hence $B(\sigma, 1/m) B(\tau, 1/m)\subseteq B(\sigma \tau, 1/n)$.

\vspace{\baselineskip}

\noindent \emph{Case 5: $\sigma, \tau \in S\setminus P_2$.}
Again, write $\sigma=u_\sigma Xv_\sigma^{-1}$ and $\tau=u_\tau Xv_\tau^{-1}$. Note that $\sigma \tau = 0$, by the presentation of $S$, regardless of the values of $v_\sigma^{-1}$ and $u_\tau$. Now let $m \in \N$ be such that $m > |v_\sigma|+|u_\tau|$, $\mu=u_\sigma e^{m_{\mu}}f^{-n_{\mu}}v_\sigma^{-1}$ for all $\mu\in B(\sigma, 1/m)$, and $\nu= u_\tau e^{m_{\nu}}f^{-n_{\nu}}v_\tau^{-1}$ for all $\nu\in B(\tau, 1/m)$. Then as usual, $m_{\mu}, n_{\mu}, m_{\nu}, n_{\nu} \geq m$ for all $\mu\in B(\sigma, 1/m) \setminus \{\sigma\}$ and $\nu\in B(\tau, 1/m) \setminus \{\tau\}$. In particular, $n_{\nu}, m_{\nu} \geq m>|v_\sigma|+|u_\tau|$, and therefore $f^{-n_{\mu}}v_\sigma^{-1}u_\tau e^{m_{\nu}}=0$. Hence $\mu\nu=0 \in B(0, 1/n) = B(\sigma \tau, 1/n)$ for all $\mu\in B(\sigma, 1/m)$ and $\nu\in B(\tau, 1/m)$.
\end{example}

We conclude with an example of a metrizable topological semigroup $T$ containing a dense copy of $P_2$, where $|T\setminus P_2| = 2^{\aleph_0}$ and $P_2$ is discrete. 

\begin{example}
Let $e,f$ be the generators of $P_2$ as an inverse semigroup, and let $A$ denote the set of sequences $p=(p_1, p_2, \ldots)$, and $B$ denote the set of sequences $q=(\ldots, q_2^{-1}, q_1^{-1})$, where $p_i, q_i\in \{e,f\}$ for all $i$. Define $S$ to be the set of pairs $\sigma=(p, q)\in A\times B$ such that the lower asymptotic density of $e$ in $p$ is more than $1/2$  and the upper asymptotic density of $e^{-1}$ in $q$ is less than $1/2$, that is $$\liminf_{n\to \infty} \frac{|\{1\leq i\leq n : p_i=e\}|}{n} > \frac{1}{2} \ \ \ \text{and} \ \ \limsup_{n\to \infty} \frac{|\{1 \leq i\leq n : q_i^{-1}=e^{-1}\}|}{n} < \frac{1}{2}.$$ It is clear that $|S|=2^{\aleph_0}$. Also, if $(p,q), (x,y) \in S$ are any elements, with $p=(p_1, p_2, \ldots)$, $q=(\ldots, q_2^{-1}, q_1^{-1})$, $x=(x_1, x_2, \ldots)$, and $y=(\ldots, y_2^{-1}, y_1^{-1})$, then $(p_1, p_2, \ldots) \neq (y_1, y_2, \ldots)$. For the sake of brevity, we shall denote such elements $\sigma = (p,q) \in S$ by $\sigma = p_1p_2\cdots q_2^{-1} q_1^{-1}$.

We define our Hausdorff topological semigroup as $T:=P_2\cup S$, with multiplication extending the usual multiplication on $P_2$, where $\sigma \tau=0$ for all $\sigma, \tau\in S$, and where for all $x\in \{e,f\}$ and $\sigma=p_1p_2\cdots q_2^{-1} q_1^{-1} \in S$,
\begin{eqnarray*}
x\cdot \sigma &=& xp_1p_2\cdots q_2^{-1} q_1^{-1}\\
\sigma \cdot x^{-1}&=& p_1p_2\cdots q_2^{-1}q_1^{-1}x^{-1} \\
\sigma \cdot x&=&
\begin{cases}
p_1p_2\cdots q_2^{-1} &\text{if }x=q_1\\
0&\text{if } x\neq q_1
\end{cases}\\
x^{-1}\cdot \sigma &=&
\begin{cases}
p_2\cdots q_2^{-1} q_1^{-1} &\text{if }x=p_1\\
0&\text{if } x\neq p_1.
\end{cases}
\end{eqnarray*}
So that we can express all elements of $T$ in the form $p_1p_2\cdots q_2^{-1} q_1^{-1}$, we make the following convention. Given $p_1, \dots, p_n, q_1, \dots, q_m \in \{e,f\}$, let $p=(p_1, \ldots, p_n,1,1 \ldots)$ and $q = (\ldots, 1^{-1}, 1^{-1}, q_m^{-1},\ldots, q_1^{-1})$. Also let $\mathfrak{1} = (1,1, \ldots)$, $\mathfrak{1}^{-1} = (\ldots, 1, 1)$, $\mathfrak{0} = (0,0, \ldots )$, and $\mathfrak{0}^{-1} = (\ldots, 0, 0)$. Then we identify $p_1 \ldots p_nq_m^{-1} \ldots q_1^{-1} = (p,q),$ $p_1 \ldots p_n = (p,\mathfrak{1}^{-1})$, $q_m^{-1} \ldots q_1^{-1} = (\mathfrak{1},q)$, $1 = (\mathfrak{1}, \mathfrak{1}^{-1})$, and $0 = (\mathfrak{0},\mathfrak{0}^{-1})$.

Next, define $d : T\times T \to \R$ by  
\begin{equation*}
  d(\sigma, \tau) =
  \begin{cases}
    0 & \text{if } \sigma = \tau \\
   \displaystyle \frac{1}{\min \{i : p_i \neq x_i \text{ or } q_i \neq y_i\}} & \text{if } \sigma \neq \tau,
  \end{cases}
\end{equation*}
where $\sigma = p_1p_2\cdots q_2^{-1} q_1^{-1}$ and $\tau =x_1x_2\cdots y_2^{-1} y_1^{-1}$. It is not hard to verify that $d$ is a metric. For any element $\sigma = p_1p_2\dots q_2^{-1}q_1^{-1} \in S$ and any $n \in \N \setminus \{0\}$ we have $$d(\sigma, p_1\dots p_nq_n^{-1}\dots q_1^{-1}) = d(p_1p_2\dots q_2^{-1}q_1^{-1},p_1\dots p_nq_n^{-1}\dots q_1^{-1}) = \frac{1}{n+1},$$ from which it follows that $P_2$ is dense in $T$. Also for all $$\sigma = p_1 \dots p_mq_n^{-1}\dots q_1^{-1} = p_1 \dots p_m11 \dots 1^{-1}1^{-1}q_n^{-1}\dots q_1^{-1} \in P_2\setminus \{0\}$$ and all $\tau \in T \setminus \{\sigma\}$, we see that $$d(\sigma, \tau) \geq \frac{1}{\min \{n+1,m+1\}},$$ while for all $\tau \in T \setminus \{0\}$, clearly $d(0, \tau) = 1,$ from which we see that $P_2$ is discrete in $T$.

It remains to show that $T$ is a topological semigroup with respect to the topology induced by $d$. To do so, let $\sigma, \tau \in T$ be arbitrary elements, and let $U$ be an open neighborhood of $\sigma\tau$. We wish to find open neighborhoods $V$ and $W$ of $\sigma$ and $\tau$, respectively, such that $VW \subseteq U$.

If $\sigma, \tau \in P_2$, then we simply take $V = \{\sigma\}$ and $W = \{\tau\}$. Next, suppose that $\sigma, \tau \in S$, and write $\sigma=x_1x_2\cdots y_2^{-1}y_1^{-1}$, $\tau=p_1p_2\cdots q_2^{-1}q_1^{-1}$. By the definition of $S$, there must be some $n \in \N$ such that $y_n \neq p_n$. Thus taking $V = B(\sigma, 1/n)$ and $W = B(\tau, 1/n)$, we see that $\mu \nu = 0$ for all $\mu \in V$ and $\nu \in W$. Since $\sigma\tau = 0$, it follows that $VW \subseteq U$.

We may therefore assume that $\sigma \in P_2$ and $\tau \in S$ (for, the case where $\sigma \in S$ and $\tau \in P_2$ can be handled analogously). If $\sigma = 0$, then taking $V = \{\sigma\}$ and $W$ to be any open neighborhood of $\tau$ gives the desired result. Let us therefore assume that $\sigma \neq 0$, and write $\sigma = xy^{-1}$ ($x, y \in \langle e,f \rangle \cup \{1\}$) and $\tau = p_1p_2\cdots q_2^{-1}q_1^{-1}$. If $\sigma \tau = 0$, then $y \neq 1$ and $y^{-1}p_1p_2\dots p_{|y|} = 0$, implying that $\sigma\mu=0$ for all $\mu\in B(\tau, 1/|y|)$. Hence, letting $V = \{\sigma\}$ and $W = B(\tau, 1/|y|)$, we have $VW = \{0\} \subseteq U$. Thus let us suppose that $\sigma \tau \neq 0$. We may also assume that $U = B(\sigma\tau, 1/m)$ for some $m \in \N$. Let $n \geq |y|+m$ be arbitrary, and set $V = \{\sigma\}$ and $W = B(\tau, 1/n)$. Then for all $\mu \in W$, we can write $$\mu = p_1p_2 \cdots p_nt_{n+1}t_{n+2} \cdots z_{n+2}^{-1}z_{n+1}^{-1}q_n^{-1} \cdots q_2^{-1}q_1^{-1}$$ for some $t_i, z_i \in \{e,f\} \cup \{1\}$. Hence for all such $\mu$, $$\sigma\mu= xy^{-1}\mu = xp_{|y|+1}p_{|y|+2} \cdots p_nt_{n+1}t_{n+2} \cdots z_{n+2}^{-1}z_{n+1}^{-1}q_n^{-1} \cdots q_2^{-1}q_1^{-1},$$ since $n > |y|$. Therefore $|xp_{|y|+1}\cdots p_n| > n-|y| \geq m$, and so $\sigma \mu \in B(\sigma\tau, 1/m) = U$. It follows that $VW \subseteq U$, as desired.
\end{example}

\vspace{.1in}

\noindent Z.\ Mesyan, Department of Mathematics, University of Colorado, Colorado Springs, CO 80918, USA 

\noindent \emph{Email:} zmesyan@uccs.edu \newline

\noindent J.\ D.\ Mitchell, Mathematical Institute, North Haugh, St Andrews, Fife, KY16 9SS, Scotland

\noindent \emph{Email:} jdm3@st-and.ac.uk \newline

\noindent M.\ Morayne,  Institute of Mathematics  and Computer Science, Wroc\l aw University of Technology, Wybrze\.ze Wyspia\'nskiego 27, 50-370 Wroc\l aw, Poland

\noindent \emph{Email:} michal.morayne@pwr.wroc.pl \newline

\noindent Y.\ P\'eresse, University of Hertfordshire, Hatfield, Hertfordshire, AL10 9AB, UK

\noindent \emph{Email:} y.peresse@herts.ac.uk

\end{document}